\documentclass[a4paper, 10pt]{amsart}

\usepackage[margin=0.75in,headheight=15pt,footskip=25pt]{geometry}

\usepackage{cite}

\usepackage{amsmath}
\usepackage{amssymb}
\usepackage{amsthm}
\usepackage{subcaption}

\usepackage{tikz}
\usepackage{pgfplots}
\usetikzlibrary{calc,patterns,angles,quotes}
\usepackage{enumerate}

\usepackage{hyperref}
\newcommand{\doi}[1]{\href{https://doi.org/#1}{#1}}

\hypersetup{colorlinks=true,allcolors=blue}
\usepackage[usenames,dvipsnames]{pstricks}
\usepackage{siunitx,booktabs}
\theoremstyle{plain}
\newtheoremstyle{spacedplain}
  {8pt}      
  {8pt}      
  {\itshape} 
  {}         
  {\bfseries}
  {.}        
  {0.5em}    
  {}

\newtheoremstyle{spacedremark}
  {8pt}
  {8pt}
  {}         
  {}
  {\bfseries}
  {.}
  {0.5em}
  {}

\theoremstyle{spacedplain}
\newtheorem{theorem}{Theorem}[section]
\newtheorem{lemma}{Lemma}[section]

\newtheorem{assumption}{Assumption}[section]
\newtheorem{remark}{Remark}[section]
\newtheorem{definition}{Definition}[section]

\usepackage{algorithm}
\usepackage{algorithmicx} 
\usepackage[skip=6pt]{caption} 

\makeatletter
\def\@floatboxreset{\reset@font\raggedright} 
\makeatother

\numberwithin{figure}{section}
\numberwithin{equation}{section}

\usepackage{graphicx}%
\usepackage{multirow}%

\usepackage{mathrsfs}%
\usepackage[title]{appendix}%
\usepackage{xcolor}%
\usepackage{textcomp}%
\usepackage{manyfoot}%
\usepackage{booktabs}%

\usepackage{algpseudocode}%
\usepackage{listings}%

\usepackage{amssymb}
\usepackage{subcaption}

\usepackage{tikz}
\usepackage{pgfplots}
\usetikzlibrary{calc,patterns,angles,quotes}
\usepackage{enumerate}

\usepackage{hyperref}

\hypersetup{colorlinks=true,allcolors=blue}
\usepackage[usenames,dvipsnames]{pstricks}
\usepackage{siunitx,booktabs}
\usepackage{url}

\numberwithin{figure}{section}
\numberwithin{equation}{section}

\newcommand{\tribar}{\vert\kern-0.25ex\vert\kern-0.25ex\vert}

\newcommand{\M}{\bold{M}}
\renewcommand{\S}{\bold{S}}
\newcommand{\Q}{\bold{T}}
\newcommand{\D}{\bold{D}}

\newcommand{\CC}{\mathcal{C}}
\newcommand{\calB}{\mathcal{B}}

\newcommand{\dwidehat}[1]{\widehat{\widehat{\,#1\,}}}
\newcommand{\II}{\mathcal{I}}

\newcommand{\ddiv}{\text{div}\,}
\newcommand{\Gh}{\mathcal{G}_h}

\newcommand{\Th}{\mathcal{T}_h}
\newcommand{\Sh}{\mathcal{S}_{h}}
\newcommand{\Vh}{\mathcal{V}_{h}}

\newcommand{\N}{N}
\newcommand{\MM}{M}

\newcommand{\Nh}{\mathcal{N}_h}
\newcommand{\Nhb}{\mathcal{N}_h^{\partial}}
\newcommand{\Ehh}{\mathcal{E}_h}

\newcommand{\NT}{N_0}
\newcommand{\Zh}{\mathcal{Z}_h}

\newcommand{\R}{\mathbb{R}}
\newcommand{\al}{\boldsymbol{\alpha}}

\newcommand{\be}{\boldsymbol{\beta}}

\newcommand{\bff}{\bold{f}}
\newcommand{\bfb}{\bold{b}}
\newcommand{\bfr}{\bold{r}}

\newcommand{\varthet}{\boldsymbol{\vartheta}}

\newcommand{\x}{\boldsymbol{x}}
 
\usepackage{stmaryrd}

\begin{document}

\title[AFC scheme for a time fractional convection--diffusion--reaction equation]{A stabilized scheme satisfying the discrete maximum principle for a time fractional convection--diffusion--reaction equation}

\author{Christos Pervolianakis}
\address{Institut für Mathematik, Friedrich-Schiller-Universität Jena, 07743, Jena, Germany}
\email{christos.pervolianakis@uni-jena.de}


\date{\today}

\begin{abstract}
We study a time fractional convection--diffusion--reaction equation in a bounded domain $\Omega\subset\mathbb{R}^2$. A stabilized numerical scheme satisfying a discrete maximum principle is constructed by combining the conforming linear finite element method with the algebraic flux correction method. The resulting semi-discrete scheme is nonlinear, and its well-posedness is established. Assuming nonsmooth initial data, we derive error estimates for the semi-discrete scheme using energy arguments. For the temporal discretization, we employ the L1 method, obtaining a fully discrete scheme for which we prove well-posedness and the discrete maximum principle. We also present numerical experiments that validate the order of convergence
as well as we test our schemes to solutions that possess layers.
\end{abstract}


\maketitle

\section{Introduction}

We shall consider the following time fractional convection--diffusion--reaction equation where we seek function $u=u(\x,t)$ for $(\x,t)\in{\Omega}\times [0,T],$ satisfying 
\begin{equation}\label{conv_diff}
\begin{cases}
\partial^\alpha_t u  - \mu\Delta u + \bfb \cdot \nabla u + \sigma \, u = G , & \text{in }{{\Omega}}\times [0,T],\\
u = 0, & \text{on }\partial {\Omega}\times[0,T],\\
u(\cdot,0)  = u_0, & \text {in }{{\Omega}},
\end{cases}
\end{equation}
with $\alpha\in (0,1),$  the diffusion coefficient $0<\mu \ll 1$, the non-negative reaction coefficient $\sigma\geq 0,$ the spatial dependent velocity field  $\bfb := \bfb(\x)\in [L^\infty(\Omega)]^2,\,\ddiv \bfb = 0.$ In addition, we assume that $G\in L^2(\Omega).$ 
Here $\partial^\alpha_t u$ denote the Caputo fractional derivative of order $\alpha\in (0,1)$ with respect to $t,$ and defined through convolution, see, e.g., \cite{kilbas2006},
\begin{align}\label{caputo_derivative}
\partial^\alpha_t u(t) := \frac{1}{\Gamma(1-\alpha)} \int_0^t (t-s)^{-\alpha}\frac{d}{ds}u(s)\,ds,
\end{align}
where $\Gamma(x) = \int_0^\infty s^{x-1}e^{s}\,ds,$ is the Gamma function.

The time--fractional evolution equations with nonsmooth initial data arise in various applications in engineering, physics, biology, finance, and thus we need to develop and  analyze efficient and accurate numerical methods for the approximation of \eqref{conv_diff}. Such applications include the electron transport in Xerox photocopier  \cite{scher1975}, thermal diffusion in fractal domains \cite{nigmatullin1986} and many more. We refer to the comprehensive surveys \cite{metzler2000,metzler2004,metzler2014} for a more complete examples with practical applications. 

A key feature for the problem \eqref{conv_diff} is that its solution has an initial layer at $t=0$ and thus $u^\prime,$ where the operator $\prime$ stands for the time derivative, may blow up when $t \to 0.$ During the past years, there a tremendous number of works proposing numerical methods addressing that feature. To name a few, we refer to  \cite{jin2016,brunner2010,yang2022,stynes2017,zhang2014,yang2024b,wang2024,jin2023}, and the references therein. Another key property of the \eqref{conv_diff} is the maximum principle satisfaction, which can be found in \cite{luchko2009} assuming a nonnegative  reaction coefficient in \eqref{conv_diff} and for a reaction coefficient of arbitrary sign in \cite{kopteva2022}.

It is known that convection--diffusion--reaction equation can possess layers, i.e., narrow regions where the solution develop steep gradients. In that regions, the standard finite element method (FEM), may produce non--physical oscillations and as a result may fail to capture the solution's sharp transition on that areas. Thus, a lot of linear and nonlinear stabilization methods has been developed. The linear methods include the streamline--upwind Petrov--Galerkin (SUPG) method \cite{brooks1982}, the so-called SOLD methods \cite{john2007}. Further, a nonlinear stabilization methods was developed, the so--called Algebraic Flux Correction schemes (AFC), which have gained a systematical attention in the last two decades, see, e.g., \cite{kuzmin2005,john2021,kuzmin2010book,barrenechea2025, barrenechea2016, barrenechea2018, barrenechea2017b, chatzipantelidis2022, barrenechea2024, jha2021, pervolianakis2024} and the references therein. The core idea of the AFC method is enforce the discrete maximum principle (DMP), and thus these spurious oscillations will be eliminated.

In the context of the time--fractional evolution equations, the DMP preservation has also drawn a lot of attention. Particularly, we refer the work \cite{jin2017}, where the authors showed that the lumped mass methods for the time--fractional heat equation preserves the non--negativity when the corresponding initial function is non--negative. In \cite{brunner2015}, the authors propose a numerical scheme based on the finite differences that satisfies the DMP. Further, in the authors in \cite{yang2024} propose an nonlinear finite volume scheme that preserves the DMP on distorted meshes. Recently, in \cite{ahmed2026}, the authors proposed a stabilized finite element scheme for solving \eqref{conv_diff} using symmetric stabilization techniques in space. They derived error estimates for smooth and nonsmooth initial data. 

\subsection{Preliminaries}

Throughout the paper, we use the standard notation for Lebesgue  and Sobolev spaces, namely we denote $W^{m}_p=W^{m}_p(\Omega)$, $H^{m}=W^{m}_{2}$, $L^{p}=L^{p}(\Omega)$, and  with $\|\cdot\|_{m,p}=\|\cdot\|_{W^{m}_p}$, $\|\cdot\|_{m}=\|\cdot\|_{H^{m}}$,  $\|\cdot\|_{L^{p}}=\|\cdot\|_{L^{p}(\Omega)}$, for $m\in\mathbb{N}$ and $p\in[1,\infty]$, the corresponding norms. In addition, we introduce the fractional order Sobolev space $\dot H^s(\Omega)\subset L^2(\Omega),\,s\geq 0,$ by
\begin{align*}
\dot H^s := \dot H^s(\Omega) = \{v\in L^2\,:\,\sum_{j=1}^\infty \lambda_j^s(v,\phi_j)^2 < \infty \},
\end{align*}
with norm
\begin{align*}
\|v\|_{\dot H^s}^2 := \sum_{j=1}^\infty (v,\phi_j)^2 < \infty,
\end{align*}
where the non--decreasing sequence $\{\lambda_j\}_{j=1}^\infty$ denotes the positive eigenvalues of $-\Delta$ with zero Dirichlet boundary conditions. In addition, the sequence $\{\phi_j\}_{j=1}^\infty$ denotes the corresponding eigenfunctions which form an orthogonal basis of $L^2.$ The non--negative powers of the operator $-\Delta$ are define through
\begin{align*}
(-\Delta)^sv = \sum_{j=1}^\infty \lambda_j^s(v,\phi_j)\phi_j,\;\;\text{for}\;\;s\geq 0,
\end{align*}
then we can equivalently express the norm $\|\cdot\|_{\dot H^s}$ as
\begin{align*}
\|v\|_{\dot H^s}^2 = \| (-\Delta)^{s/2}v\|_{L^2} = ((-\Delta)^sv,v)^{1/2}.
\end{align*}
We note that $\dot H^0 = L^2.$ For more details, we refer to \cite[Chapter 3]{thomee2006}, \cite[Section 1.3]{jin2023}. 

The basis for the methods studied is the variational formulation of the model problem, to find function $u(t)\in H^1_0,$ such that
\begin{align}
(\partial^\alpha_t u, v) + \calB(u, v) = (G,v), \;\;\;\forall\, v\in H^{1}_0,\label{weak_u}
\end{align}
where the bilinear form $\calB(\cdot,\cdot)\,:\,H^{1}_0\times H^{1}_0 \to \mathbb{R}$ is defined as $\calB(v,w) := (\mu\nabla v, \nabla w) + (\bfb\cdot\nabla v,w) + (\sigma v, w)$ for all $v,\,w\in H^{1}_0(\Omega).$ The bilinear for is bounded and coercive in $H^{1}_0(\Omega),$ due to $\ddiv\bfb=0.$ We define the energy norm, $\tribar v \tribar^2 := \mu\|\nabla v\|_{L^{2}}^2 + \sigma\|v\|_{L^{2}}^2,\,v\in H^{1}.$ We note that, $\|\nabla v\|_{L^{2}} \leq \mu^{-1/2}\tribar v \tribar,\,v\in H^{1}.$

Furthermore, we recall the regularity properties of \eqref{conv_diff}. According to \cite[Theorems 11, 12, 13]{mustapha2020}, \cite[Theorem 3]{mustapha2021}, the following regularity results holds when $G=0.$ For $0\le \delta \leq 2,$ the solution $u$ and its time derivative, satisfies for $t>0,$
\begin{align}\label{regularity}
\|u(t)\|_{\dot H^s} + t\|u^\prime(t)\|_{\dot H^s} \leq C t^{-\alpha(s-\delta)/2}\|u_0\|_{\dot H^\delta},\;\;\;0\leq s-\delta \leq 2.
\end{align}
We also assume that $G$ is regular enough so that \eqref{regularity} remains valid.

\subsection{Finite element method}

The finite element methods studied are based on triangulations $\Th =\bigcup_{i=1}^{\N_E}K_i$ with $\N_E\geq1$ of $\Omega,$ where $h = \max_{1\leq i\leq \N_E}\mathrm{diam}(K_i).$ 
 We use the finite element spaces
\begin{equation}\label{fem_space}
\Vh : = \left\lbrace \chi \in\mathcal{C}(\overline{{\Omega}})\,:\,\chi|_{K} \in \mathbb{P}_1,\;\;\forall\,K\in\Th\right\rbrace,\;\;\;\;\Sh := \Vh\cap H^1_0(\Omega),
\end{equation} 
where $\Nh = \lbrace Z_j\rbrace_{j=1}^{\MM}$ be the set of the nodes in the triangulation $\Th$ and $\N<\MM,$ are the interior nodes. Further, $\lbrace \phi_j \rbrace_{j=1}^{\MM}$ the basis functions of $\Vh$, with $\phi_j(Z_i)=\delta_{ij},$ thus, the corresponding nodal basis $\lbrace \phi_j \rbrace_{j=1}^{\N}\subset \Sh.$

The semi--discrete approximation of the variational problem \eqref{weak_u}, may be written as follows: Find function $u_h\in \Sh$, with $u_h(0) = u^0_h\in \Sh$, such that
\begin{align}
(\partial^\alpha_t u_h, \chi) + \calB(u_h, \chi)  = (G,\chi), \;\;\;\forall\, \chi\in \Sh.\label{semi_weak}
\end{align}
Let $\be = \be(t)$ with $\be=(\beta_1,\dots,\beta_{\N})^T$ be the coefficient vector, with respect to the basis of $\Sh$ of $u_h\in\Sh.$ Then, the resulting stabilized semi--discrete scheme can be written in the following form,
\begin{equation}\label{Semi_discrete_matrix_initial}
\begin{aligned}
\M\partial_t^\alpha\be(t) +  (\mu\S + \Q + \D + \sigma\,\M)\be(t) =  \bfr(G(t)),
\end{aligned}
\end{equation}
where the matrix $\M=(m_{ij})$ with elements  $m_{ij}=(\phi_i,\phi_j),\,i,j=1,\ldots,\N,$ and $\S=(s_{ij})$ with elements $s_{ij}=(\nabla\phi_i,\nabla\phi_j),\,i,j=1,\ldots,\N,$ are the usual mass and stiffness matrix, respectively. The matrix due to convection term is defined $\Q_{\al}=(\tau_{ij}),\,i,j=1,\ldots,\N,$ with elements $\tau_{ij} = (\bfb \cdot \nabla \phi_j,\phi_i),\,i,j=1,\ldots,\N.$

\subsection{Contributions and outline of the paper}

Our aim is to construct a semi-discrete scheme for a time-fractional convection--diffusion--reaction equation that satisfies a discrete maximum principle. To this end, we employ the conforming linear finite element method for the spatial discretization and enforce the discrete maximum principle through an algebraic flux correction technique. 

Due to the nonlinear nature of the algebraic flux correction, the resulting semi-discrete scheme is nonlinear. We therefore establish its existence and uniqueness by means of a fixed-point argument for sufficiently small mesh size $h$. Further, we prove that the proposed stabilized semi--discrete scheme satisfies the discrete maximum principle.

Since, in many applications, the initial data may possess only $L^2$ regularity rather than the stronger $H^2$ regularity, in order to derive error estimates, we adopt the novel energy argument approach developed in \cite{goswami2011} for linear evolution equations and in \cite{mustapha2018,mahata2022,ahmed2026} for time fractional diffusion equations. The main difference between our error analysis and those in the latter works is the presence of a nonlinear stabilization term. Consequently, the main novelty of our analysis lies in the treatment of this term throughout the error estimates, which enables us to derive optimal convergence rates with respect to $h$.
 In particular, under the regularity as in \eqref{regularity}, we derived an optimal error estimate for the AFC semi--discrete scheme for $0<\alpha< \frac{1}{2-\delta},$ of the form
\begin{align*}
\|u(t) - u_h(t)\|_{L^2} \leq C_{\alpha,\mu}h^2t^{-\alpha(2-\delta)/2}\|u_0\|_{\dot H^\delta},\;\;\;0< t \leq T.
\end{align*} 

For the fully--discrete scheme we consider the L1 discretization on a uniform temporal mesh of the Caputo temporal derivative. Again, the resulting scheme is nonlinear and we employ a fixed point argument to show its existence and uniqueness. Further, we prove the local and global discrete maximum principle satisfaction by the stabilized full--discrete scheme. 

The paper is organized as follows: In Section \ref{section:preliminaries}, we introduce some useful results from the fractional calculus. In Section \ref{section:AFC}, we employ the AFC method to the \eqref{semi_weak} and we introduce some assumptions on the correction factors for the AFC, namely we assume that the correction factors are linearity preserving and and satisfy a local Lipschitz-type estimate. Under these assumptions, we establish the well--posedness of AFC semi--discrete scheme. Further, we prove the DMP preservation. In Section \ref{section:error_analysis}, under the regularity assumptions \eqref{regularity} to the solution and the assumptions on the correction factors, we derived error estimates. In Section \ref{section:fully}, we consider the fully--discrete scheme based on L1 discretization on a uniform temporal mesh of the Caputo temporal derivative. We establish its well--posedness and we DMP preservation. Finally, in Section \ref{section:numerical_results}, we present numerical experiments performed several numerical experiments for problems involving solutions with layers.

\section{Preliminaries}\label{section:preliminaries}

In the current section, we recall some properties of the fractional integral operators. For a more detailed introduction, we refer to the book of \cite{jin2023}. For our analysis purposes, we recall the Riemann--Liouville fractional integral operator $\II^\alpha$ of order $\alpha\in (0,\infty)$ which defined as
\begin{align}\label{riemann_liouville}
\II^\alpha \varphi(t) := \int_0^t \omega_{\alpha}(t-\tau)\varphi(\tau)\,d\tau,\;\;\;\omega_{\alpha}(t):= \frac{t^{\alpha-1}}{\Gamma(\alpha)}.
\end{align}
For the the Riemann--Liouville fractional integral operator $\II^\alpha$ and the caputo fractional derivative of order $\alpha\in (0,1)$ we have that $\partial_t^\alpha \varphi(t) := \II^{1-\alpha}\varphi'(t).$ For the integral operators $\II^\alpha$ and $\II^\beta,$ for some $\alpha,\,\beta\in (0,\infty),$ we obtain the following semigroup property
\begin{align*}
\II^{\alpha}\II^{\beta} = \II^{\alpha+\beta},\;\;\;\alpha,\,\beta\in (0,\infty).
\end{align*}
In addition, some properties for the fractional derivative operator $\II^\alpha,$ that will be used in our analysis are the following. It is known, see, e.g., \cite[Lemma 3.1 (ii)]{mustapha2014} that for piecewise time continuous functions $\varphi\,:\,[0,T] \to L^2,$ it holds
\begin{align}\label{coercivity}
\int_0^T (\II^\alpha \varphi, \varphi)\,dt \geq \cos(\alpha \pi/2)\int_0^T \|\II^{\alpha/2}\varphi\|^2_{L^2}\,dt\geq 0,\;\;0<\alpha<1.
\end{align}
Furthermore, by \cite[Lemma 3.1 (iii)]{mustapha2014} and the inequality $\cos(\alpha\pi/2)\geq 1-\alpha,$ the following continuity property holds for $\II^\alpha$. For any $\varphi,\,\psi\in L^2((0,T);L^2),$ we have
\begin{align}\label{derivative_cont}
\int_0^t (\II^{1-\alpha}\varphi, \psi)\,ds \leq \epsilon \int_0^t (\II^{1-\alpha}\varphi,\varphi)\,ds + \frac{1}{4\epsilon(1-\alpha)^2}\int_0^t (\II^{1-\alpha}\psi,\psi)\,ds,\;\epsilon>0.
\end{align}
We will use $\varphi_i(t) = t^i\varphi(t),\,i=1,2.$ Then, see \cite[Lemma 2.2]{karaa2018}, \cite[Lemma 2.1]{mustapha2018} the following identities are holds for $0\leq t\leq T$ and $0<\alpha<1,$
\begin{align}
t\II^\alpha\varphi(t) & = \II^\alpha\varphi_1(t) + \alpha \II^{\alpha+1}\varphi(t),\label{aux_eq_1}\\
t\II^\alpha\varphi^\prime(t) & = \II^\alpha\varphi^\prime_1(t) + (\alpha-1) \II^{\alpha}\varphi(t) - t\omega_\alpha(t)\varphi(0),\label{aux_eq_2}\\
t^2\II^\alpha\varphi^\prime(t) & = \II^\alpha\varphi_2^\prime(t) + 2(\alpha-1) \II^{\alpha}\varphi_1(t) + \alpha(\alpha-1) \II^{\alpha}\varphi(t) - t^2\omega_\alpha(t)\varphi(0)\label{aux_eq_3}.
\end{align}
A key role on our energy estimates, the following inequality. 
\begin{lemma}\normalfont{\cite[Lemma 1]{alikhanov2012}}\label{lemma:derivative_caputo_inequality}
Let $v(t)$ an absolutely continuous function on $[0,T],$ then it holds that
\begin{align*}
v(t) \partial_t^\alpha v(t) \geq \frac{1}{2}\partial_t^\alpha v^2(t),\;\;\;0<\alpha<1.
\end{align*}
\end{lemma}

\subsection{Mesh assumptions}

We consider a  family of regular triangulations $\Th$ of the domain $\overline\Omega\subset\R^2$. We will assume that the family $\Th$ satisfies the following assumption.
 
\begin{assumption}\label{mesh-assumption}
Let $\Th =\bigcup_{i=1}^{\N_E}K_i,\,\N_E\geq1,$ be a  family of regular triangulations  of $\overline\Omega$ such that any edge of any $K$ is either a subset of the boundary $\partial\Omega$ or an edge of another $K \in \Th$, and in addition
\begin{enumerate}
\item  $\Th$ is shape regular, i.e, there exists a constant $\varpi >0,$ independent of $K$ and $\Th,$ such that 
\begin{equation}\label{shape_regularity}
\frac{h_K}{\varrho_K} \leq \varpi,\quad \forall K\in\Th,
\end{equation}
where $h_K$ the longest edge of $K,\,\varrho_K=\mathrm{diam}(B_K)$, and $B_K$ is the inscribed ball in $K$.
\item The  family of triangulations $\Th$ is quasi-uniform, i.e., there exists constant $\varrho>0$ such that
\begin{align}\label{quasi-uniformity}
\frac{\max_{K\in\Th}h_K}{\min_{K\in\Th}h_K} \leq \varrho,\quad\forall K\in\Th.
\end{align}
\item
We assume that $\Th$ satisfies an acute condition, i.e., all interior angles of a triangle $K\in\Th$ are less than or equal $\pi/2$. 
\end{enumerate}
\end{assumption}

\begin{remark}\label{remark:s_elements}
The third part of the Assumption \ref{mesh-assumption} implies that $s_{ij}\le0,\,j\neq i$ and $s_{ii}>0,$ see, e.g., {\normalfont\cite{draganescu2004}}. The importance of this assumption is highlighted in Remark \ref{remark:optimal2}.
\end{remark}

Since $\Th$ satisfies \eqref{quasi-uniformity}, we have for all $\chi\in \Sh,$ cf., e.g., \cite[Chapter 4]{brenner2008}, 
\begin{equation}\label{eq:inverse_estimate}
\|\chi\|_{L^{\infty}} + \|\nabla \chi\|_{L^{2}}  \le Ch^{-1}\|\chi\|_{L^{2}}
\text{ and } \|\nabla \chi\|_{L^{\infty}}  \le Ch^{-1}\|\nabla\chi\|_{L^{2}}.
\end{equation} 

Let $\Nh$ be the the indices of all the nodes of $\Th$, i.e., $\Nh:=\{ i: Z_i \text{ a node of the triangulation } \Th\}$ where $\Nh = \Nh^0\cup\Nhb,$ with $\Nh^0$ denotes the set of internal nodes and similarly $\Nhb,$ the set of boundary nodes. We denote  $\omega_i$, $i\in \Nh$, the collection of triangles with a common vertex $Z_i$,  i.e. $\omega_i = \cup_{Z_i\in K}\overline{K},$ see Fig. \ref{fig:patches}. The sets $\Nh(\omega)$ and $\Ehh(\omega)$ contain the vertices and  the edges, respectively, of a  subset of $\omega\subset\Th$ and $\Zh^i$ the set of nodes adjacent to $Z_i$, $\Zh^i:=\{ j: Z_j\in \Nh, \text{adjacent to }Z_i\}$.

\begin{figure}
\centering
\includegraphics[scale=0.7]{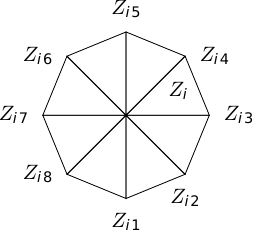} 
\caption{The subdomain $\omega_i$ of the triangulation $\Th.$ Here $Z_{i\ell},\,\ell=1,\ldots,8,$ are the neighboring nodes of $Z_i.$}\label{fig:patches}
\end{figure}

\section{Algebraic flux correction}\label{section:AFC}

\subsection{A semi--discrete scheme that satisfies the DMP}

In this section, we employ the AFC method into the semi--discrete scheme \eqref{weak_u}. As usual procedure, see, e.g., \cite{kuzmin2010book, barrenechea2025, barrenechea2024}, we replace the mass matrix $\M$ but the corresponding lumped mass matrix $\M_L,$ which is a diagonal matrix with elements $m_i = \sum_{j=1}^{\N}m_{ij}.$ In addition, the negative off--diagonal elements of $\Q$ are vanished by adding a symmetric and with zero row and column sum matrix $\D,$ so that the matrix $\Q+\D$ has non-negative off diagonal elements. We note that, since we are considering meshes that are weakly acute, the stiffness matrix has already that property, and thus, we exclude it from this.  More specifically, the matrix $\D = (d_{ij}),\,i,j=1,\ldots,\MM,$ is defined through,
\begin{equation}\label{D_def}
d_{ij} : = \max\{- \tau_{ij},0, - \tau_{ji}\}=d_{ji}\le 0,\quad\forall j\neq i,\,\,i,j=1,\ldots,\MM,\;\;\text{and}\;\;d_{ii} : = -\sum_{j\neq i}d_{ij}>0.
\end{equation}
Some of them are harmless to the DMP preservation, thus, the core idea of the AFC method is to return them proportionally.
Let $\be = \be(t)$ with $\be=(\beta_1,\dots,\beta_{\N})^T$ be the coefficient vector, with respect to the basis of $\Sh$ of $u_h\in\Sh.$ Then, the resulting stabilized semi--discrete scheme can be written in the following form,
\begin{equation}\label{Semi_discrete_matrix}
\begin{aligned}
\M_L\partial_t^\alpha\be(t) +  (\mu\S + \Q + \D + \sigma\,\M_L)\be(t) =  \bfr(G(t)) +  \overline{\mathsf{f}}(\be(t)),
\end{aligned}
\end{equation}
where $\overline{\mathsf{f}}$ denote the correction term due to the algebraic flux correction, defined as
\begin{equation}\label{correction_term}
\overline{\mathsf{f}}_{i}=\sum_{j\neq i}\mathfrak{a}_{ij}\mathsf{f}_{ij},\;\;\;\;i=1,\ldots,\N.
\end{equation}
where the correction factors $\mathfrak{a}_{ij} = \mathfrak{a}_{ji}\in[0,1],\;i,j=1,\ldots,\N,$ which will specified later in subsection \ref{subsection:factors}. Further, the fluxes $\mathsf{f}_{ij},\,i,j=1,\ldots,\N,$  are defined as follows. We assume $\bff=(f_1,\ldots,f_\N)^T$ denote the error due to the artificial diffusion operator $\D,$ defined in \eqref{D_def}, which is given by $\bff(\be) =  \D\be.$ Since the matrix $\D$ has zero sum, the latter can be decomposed into internodal fluxes, see, e.g., \cite{kuzmin2010book, barrenechea2025}, i.e.,
\begin{equation}\label{internodal_fluxes_f}
f_i = \sum_{j\neq i}\mathsf{f}_{ij},\quad i = 1,\dots,\N,\quad\text{with}\quad \mathsf{f}_{ji} =  d_{ij}\left(\beta_{j} - \beta_i\right) = -\mathsf{f}_{ij}.
\end{equation}
For the rest of this paper we will call the internodal fluxes as anti-diffusive fluxes. Every anti-diffusive flux $\mathsf{f}_{ij}$ is multiplied by a solution--depended correction factor $\mathfrak{a}_{ij}\in[0,1]$, before it is inserted into the equation. This procedure results \eqref{Semi_discrete_matrix}.

We recall that the bilinear form  $(\cdot,\cdot)_h,$ see, e.g., \cite[Chapter 15]{thomee2006}, which is an inner product in $\Sh$ that approximates $(\cdot,\cdot)$ and is defined by
\begin{equation}\label{quadrature}
(\psi,\chi)_h = \sum_{K\in\Th}Q_h^K(\psi\chi),\ \text{ with }Q_h^K(g) = \frac{1}{3}|K|\sum_{j=1}^3g(Z_j^K)\approx \int_K g\,dx,
\end{equation}
with $\{Z_j^K\}_{j=1}^3$ the vertices of a triangle $K\in\Th.$ For the inner product $(\cdot,\cdot)_h$ introduced in \eqref{quadrature}, the following holds. It is known that $(\cdot,\cdot)_h$ induces an equivalent norm to $\|\cdot\|$ on $\Sh$. In particular, exists constants $C_1,\,C_2$ independent on $h$, such that
\begin{equation}\label{mass_lump_equivalence}
C_1\|\chi\|_h  \leq \|\chi\|_{L^2}  \leq C_2\|\chi\|_{h},\ \text{ with }\ \|\chi\|_h = (\chi, \chi)_h^{1/2},\quad\forall\chi\in \Sh.
\end{equation}

To write the variational formulation of \eqref{Semi_discrete_matrix}, we recall the following bilinear form from \cite{barrenechea2016}. In particular, for a $\psi\in \Sh,$ its nodal values as are denoted by $\psi_i = \psi(Z_i),\;i=1,\ldots,\N.$ We define the bilinear form $\widehat d_{\D,h}(\psi;\cdot,\cdot)\,:\,{\CC}\times {\CC}\to{\R},$ with $\psi\in\CC$  as
\begin{align}
\widehat{d}_{\D,h}(\psi;\zeta,\chi) & := \sum_{i<j}d_{ij}(1 - \mathfrak{a}_{ij}(\psi))(\zeta_i - \zeta_j)(\chi_i - \chi_j),\label{stab_term_D_widehat}
\end{align} 
where the last equality in both bilinear forms, is due to the symmetry of matrix $\D,$ respectively, see, e.g., \cite{barrenechea2018}.

Then, the variational formulation of \eqref{Semi_discrete_matrix}, reads: We seek $u_h\in\Sh,$ such that 
\begin{equation}\label{semi_fem_u_2D_afc}
\begin{aligned}
(\partial^\alpha_t u_h, \chi)_h & + \calB_h(u_h,\chi) - \widehat{d}_{\D,h}(u_h;u_h,\chi) = (G,\chi),\;\;\forall\chi\in \Sh
\text{ with } u_h(0)  = u_h^0 := P_hu_0,
\end{aligned}
\end{equation}
where the bilinear form $\calB_h(\cdot,\cdot)\,:\,\Sh\times\Sh \to\mathbb{R},$ as $\calB_h(\psi,\chi):= (\mu\nabla \psi, \nabla \chi) + (\bfb\cdot \nabla \psi, \chi) + (\sigma\psi,\chi)_h$ for all $\psi,\,\chi\in\Sh.$ By defining the mesh--dependent energy norm, $\tribar v \tribar^2_h := \mu\|\nabla v\|_{L^{2}}^2 + \sigma\|v\|_h^2,\,v\in H^{1}.$ Note that the latter bilinear form is bounded and coercive on $\Sh$.
\begin{remark}\label{remark:h_norm}
Since the norm $\|\cdot\|_h$ is equivalent to $\|\cdot\|_{L^{2}}$ on $\Sh,$ see \eqref{mass_lump_equivalence}, it follows that $\tribar  \cdot \tribar_h$ is equivalent $\tribar \cdot \tribar$ on $\Sh.$ In particular, it holds 
\begin{align}\label{h_norm_equivalence}
\widetilde C_1\tribar \chi \tribar_h   \leq \tribar \chi \tribar \leq \widetilde C_2 \tribar \chi \tribar_h,\quad\forall\chi\in \Sh \text{ with } \widetilde C_1 = \min\{1, C_1\},\,\widetilde C_2 = \max\{1, C_2\}.
\end{align}
\end{remark}

\begin{remark}
We note that the stabilized scheme \eqref{semi_fem_u_2D_afc} is nonlinear due to the nonlinear character of the stabilization terms.
\end{remark}

\subsection{Correction factors}\label{subsection:factors}

To ensure that the AFC scheme \eqref{Semi_discrete_matrix} satisfies the discrete maximum principle,  we choose the correction factors  $\mathfrak{a}_{ij}$ such that the sum of anti-diffusive fluxes is constrained by, see, cf. e.g., \cite{kuzmin2010book, barrenechea2025},
\begin{equation}\label{led_2D_weakened}
q_i^-\left(\beta_{i}^{\min}(t)-\beta_{i}(t)\right)\leq \sum_{j\neq i}\mathfrak{a}_{ij}\mathsf{f}_{ij}\leq q_i^+\left(\beta_{i}^{\max}(t)-\beta_{i}(t)\right),\;\;q_i^\pm\geq 0.
\end{equation}
Here, we denote with $\beta_i^{\max},\,\beta_i^{\min},$ the local maximum and minimum of $\be$ over the patch $\omega_i,$ respectively. For a detailed presentation we refer to \cite[Chapter 4]{kuzmin2010book}, \cite[Section 6]{barrenechea2024}, \cite[Chapter 10]{barrenechea2025}. Examples of correction factors can be found also in  \cite{barrenechea2016,barrenechea2018,barrenechea2024,kuzmin2010book,kuzmin2002,kuzmin2004} and the references therein.
We shall assume that the correction factors $\mathfrak{a}_{ij},$ satisfies the following two properties.
\begin{assumption}(Linearity Preservation)\label{assumption:linearity_preservation}
Let $Z_i$ an inner node of $\Th.$ The limiter $\mathfrak{a}_{ij}$ is linearity preserving, i.e., for every edge $e\in\mathcal{E}_h^0$ with endpoints $Z_i,\,Z_j,$
\begin{align*}
\mathfrak{a}_{ij}(v) = 1\;\;\text{ if }\;\;v\in \mathbb{P}_1(\mathbb{R}^2).
\end{align*}
\end{assumption}
\begin{assumption}(Local Estimate)\label{assumption:local_estimate_factors}
Let $Z_i$ an inner node of $\Th.$ We assume that for every edge $e\in\mathcal{E}_h^0,$ with endpoints $Z_i,\,Z_j$, the following estimate holds for the limiter $\mathfrak{a}_{ij},$
\begin{align}\label{local_estimate}
\vert d_{ij} \mathfrak{a}_{ij}(\chi)(\chi_i - \chi_j) & - d_{ij} \mathfrak{a}_{ij}(\widetilde\chi)(\widetilde\chi_i - \widetilde\chi_j) \vert \leq C\,h\sum_{\ell\in \Zh(\omega_i)}|\chi_\ell - \widetilde\chi_\ell|,
\end{align}
for all $\chi,\,\widetilde\chi\in \Sh.$ The constant $C$ is independent of $h.$
\end{assumption}
We recall the algorithm to compute the correction factors, proposed in \cite[Lemma 6]{barrenechea2018}, see Algorithm \ref{algorithm-1}. 

\noindent
 
\begin{algorithm}
Given data: 
\begin{enumerate}
\item The positive  coefficients $q_{i}^\pm$, $i,j=1,\dots,\N,$
\item The fluxes $\mathsf{f}_{ij}$, $i\neq j$, $i,j=1,\dots,\N.$
\item The coefficients $\beta_j,$  $j=1,\dots,\N$.
\end{enumerate}
\noindent
Computation of factors $\mathfrak{a}_{ij},$ for $i,\,j\in\Nh,$ as follows.
\begin{enumerate}
\item
Compute for $i\in\Nh^0,\,j\in\Nh,$ the sums $P_{i}^{\pm}:= P_{i}^{\pm}(\al)$ of positive and negative anti-diffusive fluxes
\begin{align*}
P_{i}^{+} = \sum_{j\in \Zh^i}\max\{0, \mathsf{f}_{ij}\},
\quad \text{ and }\quad P_{i}^{-} = \sum_{j\in \Zh^i}\min\{0, \mathsf{f}_{ij}\}.
\end{align*}
\item
Compute for $i\in\Nh^0,\,j\in\Nh,$ the local extremum diminishing upper and lower bounds  ${Q}_{i}^{\pm} : = {Q}_{i}^{\pm}(\be),$
\begin{align*}
{Q}_{i}^{+} = q_i^+(\beta_{i}^{\max} - \beta_{i}),\quad\text{ and }\quad{Q}_{i}^{-} = q_i^-(\beta_{i}^{\min} - \beta_{i}),
\end{align*}
where $\beta_{i}^{\max},\;\beta_{i}^{\min}$ are the local maximum and local minimum at $\omega_i.$
\item
Compute for $i\in\Nh^0,\,j\in\Nh,$ also the coefficients $\overline{\mathfrak{a}}_{ij},$ for $j\neq i$ are given by
\begin{equation}\label{correction_factors_definition}
\begin{aligned}
R_{i}^+=\min\left\lbrace 1, \frac{Q_{i}^+}{P_{i}^+} \right\rbrace,\quad R_{i}^-= \min\left\lbrace 1, \frac{Q_{i}^-}{P_{i}^-}\right\rbrace\quad \text{and}\quad \overline{\mathfrak{a}}_{ij} = 
\begin{cases}
R_{i}^+ , &\text{if} \quad\mathsf{f}_{ij} > 0,\\
1, &\text{if} \quad\mathsf{f}_{ij} = 0,\\
R_{i}^-, &\text{if} \quad\mathsf{f}_{ij} < 0.
\end{cases}
\end{aligned}
\end{equation}
If $P_{i}^{\pm} = 0,$ then we set $R_{i}^{\pm} = 1.$
\end{enumerate}  
Then, the coefficients $\mathfrak{a}_{ij},$ for $j\neq i$ with $i\in\Nh^0,\,j\in\Nh,$ are given by $\mathfrak{a}_{ij} = \min\{\overline{\mathfrak{a}}_{ij}, \overline{\mathfrak{a}}_{ji}\}$ and $\mathfrak{a}_{ji} = \mathfrak{a}_{ij}.$ When both nodes are Dirichlet nodes, i.e., $i\in\Nh^b,\,j\in\Nh^b,$ we set $\mathfrak{a}_{ij} = 1.$
\caption{{\normalfont\cite[Lemma 6]{barrenechea2018}} Computation of the correction factors}\label{algorithm-1}
\end{algorithm}

\begin{remark}\label{remark:linearity_preservation}
The linearity preservation of Algorithm \ref{algorithm-1}, see, e.g., Assumption \ref{assumption:linearity_preservation} is valid under a sufficient choice of the coefficients $q_{i}^{\pm},\,i\in\Nh^0.$ More specifically, we consider the subsets $\Sigma_+^i,\,\Sigma_-^i$ of $\Zh^i,$ for given $i,$ defined by $\Sigma_+^i = \{j\in\Nh: j\in\Zh^i,\;\text{with}\;\alpha_i > \alpha_j\}$ and $\Sigma_-^i = \{j\in\Nh: j\in\Zh^i,\;\text{with}\;\alpha_i < \alpha_j\}.$ Then, we assume that the coefficients have the following form
\begin{align}\label{def:q_i-1}
q_i^\pm : = \gamma_i \sum_{j\in\Sigma_\pm^i}d_{ij},\quad i\in\Nh^0,\;\;\text{where}\;\;\gamma_i=\dfrac{\max_{Z_j\in \partial\omega_i}|Z_i-Z_j|}{\mathrm{dist}(Z_i,\partial \omega_i^{\mathrm{conv}})},
\end{align}
with $\omega_i^{\mathrm{conv}}$  the convex hull of $\omega_i$, for $i\in \Nh^0.$
The above choice of $q_i^\pm$ guaranties the linearity preservation of the Algorithm \ref{algorithm-1}, see, e.g.,  {\normalfont\cite[Section 4.2, Lemma 10.39]{barrenechea2018}}.
In {\normalfont\cite[Lemma 6.1]{barrenechea2017b}}, {\normalfont\cite[Lemma 10.40]{barrenechea2018}}, the authors proved that if $\omega_i$ is symmetric with respect to $Z_i,$ then $\gamma_i=1.$ 
\end{remark}

\begin{remark}
The local estimate in the Assumption \ref{assumption:local_estimate_factors} is also valid for the Algorithm \ref{algorithm-1}, see {\normalfont\cite[Lemma 10.37]{barrenechea2025}}.
\end{remark}

\subsection{Auxiliary results}

The stabilization term defined in \eqref{stability_stab_term_widehat_D} is a non-negative symmetric bilinear form which satisfies the Cauchy-Schwartz's inequality, i.e.,
\begin{align} 
\vert \widehat d_{\D,h}(\psi;\eta,\chi)\vert ^2\leq \vert \widehat d_{\D,h}(\psi;\eta,\eta)\vert \,\vert \widehat d_{\D,h}(\psi;\chi,\chi)\vert ,\quad\forall \psi,\,\eta,\,\chi\in\CC.\label{Schwartz_ineq_afc_2D}
\end{align}  
It is known, that $\widehat d_{\D,h}(\psi;\chi,\chi) \leq 0$ for all $\psi,\,\chi\in \CC$. This is due to \cite[Lemma 1]{barrenechea2016}. For the latter stabilization term, we recall some important estimates from \cite{barrenechea2016} and \cite{chatzipantelidis2022}, respectively. 

\begin{lemma}\normalfont{\cite[Lemma 16]{barrenechea2016}}\label{lemma:estimate_stab_low}
For every $\psi,\,\chi\in\Sh,$ there are exists a positive constant $C,$ independent of $h,$ such that,
\begin{align*}
\vert\widehat d_{\D,h}(\psi;\psi,\chi)\vert \leq Ch\|\bfb\|_{L^{\infty}}\|\nabla \psi\|_{L^{2}}\|\nabla \chi\|_{L^{2}},
\end{align*}
where the constant $C,$ depend on the shape regularity constant $\varpi,$ see \eqref{shape_regularity}.
\end{lemma}

\begin{lemma}\label{lemma:stability_widehat_D_main}
Let $u$ be the solution of the time--fractional problem \eqref{conv_diff} satisfying the regularity estimate \eqref{regularity}. In addition, let the correction factors satisfies Assumptions \ref{assumption:linearity_preservation}, \ref{assumption:local_estimate_factors}, then, there are exists a positive constant $C,$  independent of $h,$ such that for all $\psi,\,\chi\in \Sh,$
\begin{align}
\vert\widehat d_{\D,h}(\psi;\psi,\chi)\vert  & \leq Ch(\|\nabla (\psi - u)\|^2_{L^{2}} + h^2t^{-\alpha(2-\delta)}\|u_0\|^2_{\dot H^\delta})^{1/2}\|\nabla \chi\|_{L^{2}}.\label{stability_stab_term_widehat_D}
\end{align}
\end{lemma}
\begin{proof}
The estimate \eqref{stability_stab_term_widehat_D} was derived in \normalfont{\cite[Lemma 3.4]{chatzipantelidis2022}}, where $H^2$ regularity of the function $u,$ was assumed. Here, based on the estimates \eqref{ritz_projection_est2_2D}, we extent the result when the solution of the time--fractional problem \eqref{conv_diff} satisfies the regularity estimate \eqref{regularity}. 

We adapt here the ideas of the proof \normalfont{\cite[Lemma 3.4]{chatzipantelidis2022}} as well as the regularity estimate \eqref{regularity}. The stabilization term \eqref{stab_term_D_widehat} can be written as
\begin{align}\label{lemma:stability_eq1}
\widehat d_{\D,h}(\psi;\psi,\chi) & = \sum_{i<j}d_{ij}\widetilde{\rho}_{ij}(\psi)(\chi_i - \chi_j),
\end{align}
where $\widetilde\rho_{ij}(\psi) : = (1 - \mathfrak{a}_{ij}(\psi))(\psi_i - \psi_j),\,\psi\in\Sh.$ 
In view of the Assumption \ref{assumption:linearity_preservation} and similar to \cite[Lemma 3.4]{chatzipantelidis2022}, we define a local problem so that $\widetilde\rho_{{ij}} = 0$ is zero for some polynomial solution.
In particular, we define the following local problem on $\omega_i,\,i\in\Nh^0.$
For $u$ solution of \eqref{conv_diff}, we let $\zeta^i\in \mathbb{P}_1(\omega_{i})\cap H^1_0(\omega_i)$ be the unique solution of the following elliptic problem,
\begin{align*}
(\nabla \zeta^{i}, \nabla \chi)_{L^{2}(\omega_{i})} & = (u, \chi)_{L^{2}(\omega_{i})},\;\;\;\forall\,\chi\in \mathbb{P}_1(\omega_{i}).
\end{align*}
i.e., $\zeta^i\in\mathbb{P}_1(\omega_i)$ and $\zeta^i = 0,$ outside $\omega_i.$ Then, from the standard error analysis theory, we obtain that
\begin{align*}
\|\nabla (\zeta^i - u)\|_{L^2(\omega_i)} \leq Ch\|u\|_{H^2(\omega_i)} \leq C h t^{-\alpha(2-\delta)/2}\|u_0\|_{\dot H^\delta}
\end{align*}
Then, $\widetilde\rho_{{ij}}(\zeta^{i}) = 0$ and therefore,
\begin{align*}
\widehat{d}_{\D,h}(\psi;\psi,\chi) & = \sum_{i<j}d_{ij}( \widetilde{\rho}_{{ij}}(\psi) - \widetilde\rho_{{ij}}(\zeta^{i}))(\chi_i - \chi_j).
\end{align*}
We note that in view of Assumption \ref{assumption:local_estimate_factors}, we obtain
\begin{align*}
d_{ij}( \widetilde{\rho}_{{ij}}(\psi) - \widetilde\rho_{{ij}}(\zeta^{i})) \leq Ch\left(\|\nabla (\psi - u)\|_{L^{2}(\omega_i)} + ht^{-\alpha(2-\delta)/2}\|u_0\|_{\dot H^\delta(\omega_i)}\right).
\end{align*}
Using the latter result, the Cauchy-Schwartz inequality and the shape regularity of the triangulation, we conclude to \eqref{stability_stab_term_widehat_D}.
\end{proof}

\begin{remark}
The estimate \eqref{stability_stab_term_widehat_D} can be viewed as a sharper estimate to that on Lemma \ref{lemma:estimate_stab_low}, in the sense that the term $\widehat d_{\D,h}(u_h(t);u_h(t),\chi)$ that appears in the error analysis can be estimated by 
\begin{align*}
\widehat d_{\D,h}(u_h(t);u_h(t),\chi) \leq Ch\|\nabla (u_h(t) - R_hu(t))\|_{L^2}\|\nabla \chi\|_{L^2} +  Ch^2t^{-\alpha(2-\delta)/2}\|u_0\|_{\dot H^\delta}\|\nabla \chi\|_{L^2}.
\end{align*}
Then, the first term on the right hand side, contains the discrete error and can be absorbed to the left hand side of the equation, the second one is optimal with respect to $h.$
\end{remark}

\begin{remark}\label{remark:optimal2}
The matrix $\mu\S + \Q + \D$ of \eqref{Semi_discrete_matrix} has non--negative off--diagonal elements due to Assumption \ref{mesh-assumption}(iii), see also Remark \ref{remark:s_elements} and the \eqref{D_def}. Of course, one can remove the Assumption \ref{mesh-assumption}(iii) and modify the \eqref{D_def}, so that $d_{ij} =  \max\{\mu s_{ij}- \tau_{ij},0,\mu s_{ji} - \tau_{ji}\}=d_{ji}\le 0,\,\forall j\neq i,\,i,j=1,\ldots,\MM.$  Again, the matrix $\mu\S + \Q + \D$ will have non--negative off--diagonal elements, but $|d_{ij}| \leq 2\mu|s_{ij}| + |\tau_{ij}| + |\tau_{ji}|$ and thus in the diffusion--dominant regime, we will have $|d_{ij}| = \mathcal{O}(1).$ As a result, the estimate of the Lemma \ref{lemma:stability_widehat_D_main} will become
\begin{align*}
\vert\widehat d_{\D,h}(\psi;\psi,\chi)\vert  & \leq C(\|\nabla (\psi - u)\|^2_{L^{2}} + h^2t^{-\alpha(2-\delta)}\|u_0\|^2_{\dot H^\delta})^{1/2}\|\nabla \chi\|_{L^{2}},
\end{align*}
and we will not able to derive optimal error estimates in $L^2-$norm, with respect to $h.$
\end{remark}

\subsection{Well--posedness of AFC semi--discrete scheme}

In the current subsection, our goal is to prove the well--posedness of the nonlinear semi--discrete scheme \eqref{semi_fem_u_2D_afc}. We define the bilinear forms $d_{\D,h}(\cdot,\cdot)\,:\,{\CC}\times {\CC}\to{\R}$ and $\widetilde d_{\D,h}(\psi;\cdot,\cdot)\,:\,{\CC}\times {\CC}\to{\R}$ with $\psi\in\CC$ as
\begin{align}
d_{\D,h}(\zeta,\chi) & := \sum_{i<j}d_{ij}(\zeta_i - \zeta_j)(\chi_i - \chi_j),\label{stab_term_D}\\
\widetilde d_{\D,h}(\psi;\zeta,\chi) & := \sum_{i<j}d_{ij}\mathfrak{a}_{ij}(\psi)(\zeta_i - \zeta_j)(\chi_i - \chi_j).\label{stab_term_D_widetilde}
\end{align} 

\begin{remark}\label{remark:stabilization_terms}
We note that for a given $\psi \in \CC$, it holds that
\begin{align*}
\widehat d_{\D,h}(\psi;\zeta,\chi) = d_{\D,h}(\zeta,\chi) - \widetilde d_{\D,h}(\psi;\zeta,\chi),
\qquad \forall\, \zeta,\, \chi \in \CC.
\end{align*}
Furthermore, both bilinear forms defined in \eqref{stab_term_D} and \eqref{stab_term_D_widetilde} it holds that $ d_{\D,h}(\chi,\chi)\leq  0$ for all $\chi\in\CC$ and $\widetilde d_{\D,h}(\psi;\chi,\chi)\leq 0$ for all $\psi,\,\chi\in\CC$. The latter results is due to \cite[Lemma 1]{barrenechea2016}. In addition, the Cauchy--Schwarz inequality \eqref{Schwartz_ineq_afc_2D} and Lemma~\ref{lemma:estimate_stab_low} apply to them. We also note that Lemma \ref{lemma:stability_widehat_D_main} is true for the stabilization term defined in \eqref{stab_term_D_widetilde} but not for \eqref{stab_term_D}.
\end{remark}

\begin{lemma}{\normalfont{\cite[Lemma 3.5]{chatzipantelidis2022}}}\label{lemma:stability_widehat_D_difference}
Let the bilinear forms $\widehat d_{\D,h},\,\widetilde d_{\D,h}$ defined in \eqref{stab_term_D_widehat}, \eqref{stab_term_D_widetilde}, respectively. If their correction factors are satisfying Assumption \ref{assumption:local_estimate_factors}, there exists for $w,\,v,\,\chi\in \Sh,$ a positive constant $C$, independent of $h,$ such that,
\begin{align}
\vert\widehat{d}_{\D,h}(v;v, \chi) - \widehat{d}_{\D,h}(w;w,\chi)\vert + \vert\widetilde{d}_{\D,h}(v;v, \chi) - \widetilde{d}_{\D,h}(w;w,\chi)\vert & \leq Ch\|\nabla (v - w)\|_{L^{2}}\|\nabla \chi\|_{L^{2}}.\label{stability_stab_term_widetilde_D}
\end{align}
\end{lemma}
We consider the linear operator $\Gh\,:\,\Sh \to \Sh,\,z_h\mapsto \Gh(z_h)$ defined for a given $z_h\in \Sh,$ as 
\begin{align}\label{semi_fem_u_2D_afc_linear}
\begin{aligned}
& (\partial^\alpha_t \Gh v(t), \chi)_h  + \calB_h(\Gh v(t),\chi)\\
&\qquad - d_{\D,h}(\Gh v(t),\chi) + \widetilde d_{\D,h}(z_h(t); z_h(t),\chi) = (G(t),\chi),\;\forall\chi\in \Sh
\text{ with } \Gh v(0) = u_h^0,
\end{aligned}
\end{align}
The following definition summarizes the computation of the correction factors employed in the scheme \eqref{semi_fem_u_2D_afc} as well as in \eqref{semi_fem_u_2D_afc_linear}.
\begin{definition}\label{definition:correction_factors_AFC}
The correction factors in the stabilization terms in fully discrete schemes \eqref{semi_fem_u_2D_afc} and \eqref{semi_fem_u_2D_afc_linear} are computed using Algorithm \ref{algorithm-1}. More specifically, for a finite element function $\psi\in\Sh$ with coefficient vector $\varthet(t)\in\mathbb{R}^{\N},$ i.e., $\psi(t) = \sum_{j=1}^{\N}\vartheta_j(t)\phi_j,$ the correction factor $\mathfrak{a}_{ij}(\psi),\,i,j=1,\ldots,\N,$ is computed as follows. 
The $\mathfrak{a}_{ij}(\psi)$ are computed from Algorithm \ref{algorithm-1} with $Q^{\pm}(\varthet),\,P^{\pm}(\varthet)$ and $q_i = \gamma_i\sum_{j\in\Zh^i}d_{ij}.$ 
\end{definition}

\begin{remark}
We note that the nonlinear fully--discrete scheme \eqref{semi_fem_u_2D_afc} can be viewed as the fixed point of $u_h = \Gh(u_h).$ In what follows, our goal is to show that $\Gh\,:\,\Sh \to \Sh$ and $\Gh$ is contractive in $\Sh$ for sufficiently small $h.$
\end{remark}
Similarly to \eqref{Semi_discrete_matrix}, we can write \eqref{semi_fem_u_2D_afc_linear} in matrix form. Let $\widehat \be = \widehat \be(t)$ with $\widehat \be=(\widehat \beta_1,\dots, \widehat \beta_{\N})^T$ and $\widetilde \be = \widetilde \be(t)$ with $\widetilde \be=(\widetilde \beta_1,\dots, \widetilde \beta_{\N})^T,$ be the coefficient vectors with respect to the basis of $\Sh$ of $\Gh v,\,z_h\in\Sh,$ respectively. Then, the resulting linearized stabilized semi--discrete scheme can be written in the following form,
\begin{equation}\label{Semi_discrete_matrix_linear}
\begin{aligned}
\M_L\partial_t^\alpha\widehat \be(t) +  (\mu\S + \Q + \D + \sigma\,\M_L)\widehat \be(t) = \bfr(G(t)) +  \overline{\mathsf{f}}(\widetilde\be(t)),
\end{aligned}
\end{equation}
where $\overline{\mathsf{f}}$ denote the correction term as defined in \eqref{correction_term}.

\begin{theorem}[Well--posedness of \eqref{semi_fem_u_2D_afc}]\label{theorem:fixed_point_semi_discrete}
We assume that the correction factors in \eqref{semi_fem_u_2D_afc} and \eqref{semi_fem_u_2D_afc_linear} are defined through \eqref{led_2D_weakened} and Algorithm \ref{algorithm-1}. Then, for sufficiently small $h,$ the nonlinear stabilized semi--discrete scheme \eqref{semi_fem_u_2D_afc} has a unique solution.
\end{theorem}
\begin{proof}
Let $z_h^1,\,z_h^2\in\Sh$ given and $\Gh v_1,\,\Gh v_2\in\Sh,$ the solution of \eqref{semi_fem_u_2D_afc_linear} with $z_h^1,\,z_h^2\in\Sh,$ respectively. We also set $z_h := z_h^1 - z_h^2\in\Sh$ and $\Gh v : = \Gh v_1 - \Gh v_2\in\Sh,$ then in view of \eqref{semi_fem_u_2D_afc_linear}, we get
\begin{align}\label{existence_ineq}
(\partial^\alpha_t \Gh v, \chi)_h & + \calB_h(\Gh v,\chi) - d_{\D,h}(\Gh v,\chi) = \widetilde d_{\D,h}(z_h^2; z_h^2,\chi) - \widetilde d_{\D,h}(z_h^1; z_h^1,\chi),\;\;\forall\chi\in \Sh.
\end{align}
In view of Lemma \ref{lemma:stability_widehat_D_difference}, the right hand side difference can be estimated as
\begin{align*}
\vert \widetilde d_{\D,h}(z_h^1; z_h^1,\chi) - \widetilde d_{\D,h}(z_h^2; z_h^2,\chi) \vert & \leq Ch \|\nabla z_h\|_{L^2}\|\nabla \chi\|_{L^2},\;\;\forall\,\chi\in\Sh.
\end{align*}
Setting $\chi = \Gh v\in\Sh$ into \eqref{existence_ineq}, and in view of Lemma \ref{lemma:derivative_caputo_inequality},
\begin{align*}
\partial^\alpha_t\|\Gh v(t)\|_h^2 & + \tribar \Gh v(t) \tribar^2 - d_{\D,h}(\Gh v(t),\Gh v(t))  \leq Ch\|\nabla z_h(t)\|_{L^2}\|\nabla \Gh v(t)\|_{L^2}.
\end{align*}
Using the non--positivity of the bilinear form on the left hand side, see e.g., Remark \ref{remark:stabilization_terms}, the Remark \ref{remark:h_norm} and Young's inequality, we obtain
\begin{align*}
\partial^\alpha_t\|\Gh v(t)\|_h^2  + \tribar\Gh v(t)\tribar^2  \leq C\mu h^2\tribar z_h(t)\tribar^2.
\end{align*}
We operate both sides by $\partial_t^{-\alpha}$ and using the $\partial_t^{-\alpha}\partial_t^{\alpha} v(t) = v(t) -v(0),$ we obtain
\begin{align*}
\|\Gh v(t)\|_h^2 + \partial_t^{-\alpha}\tribar\Gh v(t)\tribar^2 \leq \|\Gh v(0)\|_h^2 + C\mu  h^2\partial_t^{-\alpha} \tribar z_h(t)\tribar^2.
\end{align*}
Next, since $\|\Gh v(t)\|_h^2 \geq 0$ for all $t\in(0,T]$ and $\|\Gh v(0)\|_h^2 = 0,$ we arrive to
\begin{align*}
\partial_t^{-\alpha}\tribar\Gh v(t)\tribar^2 \leq C\mu  h^2\partial_t^{-\alpha} \tribar z_h(t)\tribar^2.
\end{align*}
We operate the latter with $\partial_t^{\alpha}$ and we note that $\partial_t^\alpha\partial_t^{-\alpha}v=v$ with $v(0)=0,$ which can be easily seen from
\begin{align*}
\partial_t^\alpha(\partial_t^{-\alpha}v)(t) = \II^{1-\alpha}(\partial_t^{-\alpha}v)^\prime(t) = \II^{1-\alpha}(\II^{\alpha}v)^\prime(t) = \II^{1-\alpha}(\II^{\alpha-1}v)(t) = v(t).
\end{align*}
Then,
\begin{align*}
\tribar\Gh v(t)\tribar^2 \leq C\mu  h^2\tribar z_h(t)\tribar^2.
\end{align*}
Therefore, for sufficiently small $h>0,$ such that $C\mu h^2<1,$ the mapping $\Sh\ni z_h\mapsto \Gh(z_h)\in\Sh$ is contractive and thus exists a unique solution $u_h\in\Sh$ to \eqref{semi_fem_u_2D_afc}.
\end{proof}

\subsection{Discrete Maximum Principle}

In this section, we turn our attention to the discrete maximum principle (DMP) satisfaction by the semi--discrete scheme \eqref{Semi_discrete_matrix}.  We first recall the definition of the global DMP for the scheme \eqref{Semi_discrete_matrix}.

\begin{definition}(Global DMP for semidiscrete)\label{definition:DMP_semi}
We say that the solution $\be(t) \in \mathbb{R}^{\N},\,t\in[0,T]$ of \eqref{Semi_discrete_matrix} with $\sigma>0$ in $\Omega$ satisfies the DMP if for $i\in\Nh^0,$ we have
\begin{align}
G(t) & \leq 0  \;\;\text{in}\;\;\Omega  \Rightarrow \;\beta_i(t) \leq \beta^{\max,+}(0),\label{local_semi_dmp1}\\
G(t) & \geq 0  \;\;\text{in}\;\;\Omega  \Rightarrow \;\beta_i(t) \geq \beta^{\min,-}(0),\label{local_semi_dmp2}
\end{align}
where 
\begin{align*}
\beta^{\max,+}(\tau) &:= \max\left\{0, \beta^{\max}(\tau)\right\},\;\;\text{with}\;\;\beta^{\max}(\tau) := \max_{j\in\Nh} \beta_j(\tau),\\
\beta^{\min,-}(\tau) &:= \min\left\{0, \beta^{\min}(\tau)\right\},\;\;\text{with}\;\;\beta^{\min}(\tau) := \min_{j\in\Nh} \beta_j(\tau).
\end{align*}
If $\sigma = 0$ in $\Omega,$ then,
\begin{align}
G(t) & \leq 0  \;\;\text{in}\;\;\Omega  \Rightarrow \;\beta_i(t) \leq \beta^{\max}(0),\label{local_semi_dmp1_2}\\
G(t) & \geq 0  \;\;\text{in}\;\;\Omega  \Rightarrow \;\beta_i(t) \geq \beta^{\min}(0).\label{local_semi_dmp2_2}
\end{align}
\end{definition}

To prove the DMP in the semidiscrete case, in the sense of the Definition \ref{definition:DMP_semi}, we recall a result from \cite[Theorem 1]{luchko2010}, \cite[Theorem 1]{luchko2010}, \cite[Lemma 3.3]{brunner2015}. We introduce the function space $\CC^{(\alpha)}([0,T])$ where $0<\alpha<1,$ defined by
\begin{align*}
\CC^{(\alpha)} := \{v\,:\, v\in \CC([0,T])\;\;\text{and}\;\;\partial_t^\alpha v\in \CC([0,T])\}.
\end{align*}

\begin{lemma}{\normalfont{\cite[Lemma 3.1]{brunner2015}}}\label{lemma:caputo_derivative_form}
Let $\widehat\alpha$ such that $0<\alpha<\widehat \alpha \leq 1$ and the function $v\in \CC^{0,\widehat\alpha}([0,T]).$ Then, $v\in \CC^{(\alpha)}([0,T])$ and
\begin{align}\label{caputo_derivative_form}
(\partial_t^\alpha v)(t) = \frac{1}{\Gamma(1-\alpha)}\left( \frac{v(t) - v(0)}{t^\alpha} + \alpha\int_0^t (t-\tau)^{-1-\alpha}(v(t) - v(\tau))\,d\tau\right),\;\;t\in(0,T].
\end{align}
\end{lemma}

\begin{remark}
The converse result of Lemma \ref{lemma:caputo_derivative_form}, can also be proved. Indeed for a $\alpha\in (0,1),$ if $v\in \CC^{(\alpha)}([0,T])$ then $v\in \CC^{0,\alpha}([0,T]).$ This result can be found in \cite[Lemma 3.2]{brunner2015}.
\end{remark}

\begin{lemma}{\normalfont{\cite[Lemma 3.3]{brunner2015}}}\label{lemma:caputo_derivative_on_minima}
Assume that $v(t)\in \CC^{(\alpha)}([0,T])$ for $0<\alpha<1.$ The following statements are true
\begin{enumerate}
\renewcommand{\labelenumi}{\roman{enumi})}
\item
If there is a point $t^*\in (0,T)$ such that $v(t) - v(t^*) \leq 0$ for $t\in [0,t^*],$ then $\partial_t^\alpha v(t^*) \geq 0.$
\item
If there is a point $t^*\in (0,T)$ such that $v(t) - v(t^*)\geq 0$ for $t\in [0,t^*],$ then $\partial_t^\alpha v(t^*) \leq 0.$
\end{enumerate}
\end{lemma}
\begin{proof}
This lemma was first proved in \cite[Theorem 1]{luchko2010}, under stronger regularity assumption to $v,$ namely $v\in W^{1}_t((0,T])\cap \CC([0,T]),$ where $W^{1}_t((0,T])$ denotes the space of functions $v\in \CC^1((0,T])$ with $v^\prime\in L^1((0,T)).$ Subsequently, the authors of \cite[Lemma 3.3]{brunner2015} weakened these regularity assumptions to $v(t)\in \CC^{(\alpha)}([0,T]).$ Their proof is a direct application of Lemma \ref{lemma:caputo_derivative_form}. Indeed, for (i), setting $t=t^*$ into \eqref{caputo_derivative_form}, we get that both terms are non-negative due to the hypothesis of (i). Similarly, for (ii).
\end{proof}

\begin{theorem}\label{theorem:DMP_semi_discrete}
We assume that the correction factors in \eqref{semi_fem_u_2D_afc} and \eqref{semi_fem_u_2D_afc_linear} are defined through \eqref{led_2D_weakened} and Algorithm \ref{algorithm-1}. Then, the nonlinear stabilized semi--discrete scheme \eqref{semi_fem_u_2D_afc} satisfies the global DMP.
\end{theorem}
\begin{proof}
Recall the nonlinear system of ODEs \eqref{Semi_discrete_matrix}, which according to Theorem \ref{theorem:fixed_point_semi_discrete} has a unique solution. Let $i\in\Nh^0,$ then, the $i-$th equation can be written as
\begin{align*}
m_i \partial_t^\alpha\beta_i(t) + \sum_{j=1}^{\N}\left(\mu s_{ij} + \tau_{ij} + d_{ij} \right)\beta_j(t) + \sigma m_i \beta_i(t) = r_i(G(t)) + \sum_{j\neq i}d_{ij}\mathfrak{a}_{ij}(\beta(t))(\beta_j(t) - \beta_i(t)),\;\;\;0\leq t\leq T.
\end{align*}
By using the zero row sum property of matrices $\S,\,\Q$ and $\D,$ we rearrange the terms, to get
\begin{align}\label{semi_dmp_eq2}
m_i \partial_t^\alpha\beta_i(t) + \sum_{j\neq i}\left(\mu s_{ij} + \tau_{ij} + d_{ij}(1-\mathfrak{a}_{ij}(\be(t)) \right)(\beta_j(t) - \beta_i(t)) + \sigma m_i \beta_i(t) = r_i(G(t)),\;\;\;0\leq t\leq T.
\end{align}
We prove only the \eqref{local_semi_dmp2}, as the remaining cases can be treated analogously. We first note that \eqref{local_semi_dmp2} holds trivially for $t=0.$

To show \eqref{local_semi_dmp2} it is sufficient to assume that at node $i,$ the local minimum of $\beta(t)$ on $\omega_i$ is attained, i.e., we assume $\beta_i(t) := \min_{j\in\Zh^i\cup \{i\}}\beta_j(t).$ Let that \eqref{local_semi_dmp2} fails before $T,$ i.e., there exists $t^* = \inf\{t\,:\, \beta_i(t) - \beta^{\min,-}(0) < 0\}$ with $t^*\in (0,T).$
By definition of $\beta_i,$ we have that $\beta_i(t^*) \leq \beta_j(t^*),\,j\in\Zh^i.$

Then, since the correction factors are computed according to Algorithm \ref{algorithm-1}, we obtain that $\mu s_{ij} + \tau_{ij} + d_{ij}(1-\mathfrak{a}_{ij}(\be(t^*)))\leq 0$ for $j\neq i.$ We note that $s_{ij}\geq 0,\,j\neq i$ in view of Remark \ref{remark:s_elements} and then, using similar arguments to the proof of \cite[Theorem 10.35]{barrenechea2025}, we conclude to the non--positivity of these elements. In addition, since the sum on the left hand side is non--positive and $r_i(G(t))\geq 0,\,t\in[0,T],$ by hypothesis, we conclude that
\begin{align}\label{semi_dmp_eq1}
\partial_t^\alpha\beta_i(t^*) + \sigma\beta_i(t^*) \geq 0.
\end{align}
On the other hand, due to above hypothesis we have that $\beta_i(t) \geq \beta^{\min,-}(0)$ for $t\in[0,t^*]$ and due to the continuity at $t^*,$ we get $\beta^{\min,-}(0) = \beta_i(t^*),$ hence $\beta_i(t) \geq \beta_i(t^*),\,t\in[0,t^*].$ According to Lemma \ref{lemma:caputo_derivative_on_minima} (ii), we get that $\partial_t^\alpha\beta_i(t^*)\leq 0$ and thus
\begin{align*}
0 > \sigma\beta_i^{\min,-}(0) = \sigma\beta_i(t^*) \geq \partial_t^\alpha\beta_i(t^*) + \sigma\beta_i(t^*) \geq 0,
\end{align*}
which contradicts to our hypothesis, and therefore \eqref{local_semi_dmp2} is true for all $t\in[0,T].$
\end{proof}

\section{Error analysis}\label{section:error_analysis}
In this section, is devoted to derive error bounds using energy arguments for the stabilized semi--discrete scheme \eqref{semi_fem_u_2D_afc} in $L^2$ norm, in case of nonsmooth initial data $u_0.$ To achieve this, we follow arguments presented in \cite{mustapha2018,mahata2022}. To perform the error analysis, we recall the elliptic projection $R_h\,:H^1_0\to\Sh$ defined for $v\in H^1_0$ as
\begin{align*}
\calB(R_hv,\chi) = \calB(v,\chi),\;\;\forall\,\chi\in\Sh.
\end{align*}
In what follows, we split as usual the error for \eqref{semi_fem_u_2D_afc_linear}, by intermediate the elliptic projection as 
\begin{align*}
u_h(t)  - u(t) = (u_h(t)  - R_hu(t)) - (u(t)-R_hu(t)) =: \theta(t) - \rho(t)
\end{align*}
The discrete error contribution $\theta(t)\in\Sh,$ satisfies the following relation for $0<t\leq T,$ and for all $\chi\in\Sh,$
\begin{align}\label{error_eq_1}
(\II^{1-\alpha}\theta^\prime(t), \chi)_h + \calB_h(\theta(t), \chi) = - (\II^{1-\alpha}\rho^\prime(t), \chi) + \widehat d_{\D,h}(u_h(t);u_h(t),\chi). 
\end{align}
Then, for the solution of \eqref{conv_diff}, under the regularity assumptions \eqref{regularity} and the bounds derived in \cite[Chapter 8]{brenner2008} and \cite{thomee2006}, the elliptic projection $R_h$ satisfy the following bounds for $0\leq \delta \leq s$ with $s=1,2,$
\begin{align}
\|\rho(t)\|_{L^{2}} + h\|\rho(t)\|_{1} & \le Ch^s t^{-\alpha(s-\delta)/2}\|u_0\|_{\dot H^\delta},\label{ritz_projection_est2_2D}\\
\|\rho^\prime(t)\|_{L^{2}} + h\|\rho^\prime(t)\|_{1} & \le Ch^s t^{-\alpha(s-\delta)/2-1}\|u_0\|_{\dot H^\delta}.\label{ritz_projection_est2_2D_2}
\end{align}
We introduce the following notation
\begin{align*}
\widehat \varphi(t) & := \II \varphi(t) = \int_0^t\varphi(s)\,ds,\\
\dwidehat{\varphi}(t) & := \II^2 \varphi(t) = \int_0^t\int_0^s\varphi(\tau)\,d\tau\,ds.
\end{align*}
As a direct result, we can estimate for latter reference the following quantities.
\begin{equation}\label{rho_integral_estimate}
\begin{aligned}
\|\II^{1-\alpha}\rho(t)\|_{L^2} + \|\II^{1-\alpha}\rho_1^\prime(t)\|_{L^2} & \leq C_\alpha \int_0^t (t-\tau)^{-\alpha}\left( \|\rho(\tau)\|_{L^2} + \tau \|\rho'(\tau)\|_{L^2}\right)\,d\tau\\
& \leq C_\alpha h^s\int_0^t(t-\tau)^{-\alpha}\tau^{-\alpha(s-\delta)/2}\|u_0\|_{\dot H^\delta}\,d\tau\\
& \leq C_\alpha h^st^{1-\alpha-\alpha(s-\delta)/2}\|u_0\|_{\dot H^\delta},\;\;\text{for}\;\;0\leq \delta\leq s,\;\;s=0,1.
\end{aligned}
\end{equation} 
Similarly, 
\begin{equation}\label{rho_integral_estimate2}
\begin{aligned}
\|\II^{1-\alpha}\rho_1(t)\|_{L^2} & \leq C_\alpha \int_0^t (t-\tau)^{-\alpha} \tau \|\rho(\tau)\|_{L^2}\,d\tau\\
& \leq C_\alpha h^2\int_0^t(t-\tau)^{-\alpha}\tau^{1-\alpha(s-\delta)/2}\|u_0\|_{\dot H^\delta}\,d\tau\\
& \leq C_\alpha h^2t^{2-\alpha-\alpha(2-\delta)/2}\|u_0\|_{\dot H^\delta},\;\;\text{for}\;\;0\leq \delta\leq 2.
\end{aligned}
\end{equation} 
In addition,
\begin{equation}\label{rho_integral_estimate3}
\begin{aligned}
\|\II^{1-\alpha}\widehat \rho(t)\|_{L^2} = \|\II^{2-\alpha} \rho(t)\|_{L^2} & \leq  C_\alpha \int_0^t (t-\tau)^{1-\alpha} \|\rho(\tau)\|_{L^2}\,d\tau\\
& \leq C_\alpha h^s\int_0^t(t-\tau)^{1-\alpha}\tau^{-\alpha(s-\delta)/2}\|u_0\|_{\dot H^\delta}\,d\tau\\
& \leq C_\alpha h^st^{2-\alpha-\alpha(s-\delta)/2}\|u_0\|_{\dot H^\delta},\;\;\text{for}\;\;0\leq \delta\leq s,\;\;s=0,1.
\end{aligned}
\end{equation} 
Further, we can estimate in a similar manner, the
\begin{align}\label{rho_integral_estimate4}
\begin{aligned}
& \int_0^t \left(  (\II^{1-\alpha}\rho_1(\tau), \rho_1(\tau))_h + (\II^{1-\alpha}\widehat \rho(\tau), \widehat \rho(\tau))_h\right)\,d\tau\\
&\qquad \leq \int_0^t \left( \|\II^{1-\alpha}\rho_1(\tau)\|_h + \|\II^{1-\alpha}\widehat \rho(\tau)\|_h\right) \left(\|\rho_1(\tau)\|_h + \|\widehat\rho(\tau)\|_h\right)\,d\tau\\
& \qquad \leq Ch^{4}\int_0^t s^{2-\alpha-\alpha(2-\delta)/2}s^{1-\alpha(2-\delta)/2}\|u_0\|_{\dot H^\delta}^2\,ds \leq C_{\alpha}h^{4}t^{4-\alpha-\alpha(2-\delta)}\|u_0\|_{\dot H^\delta}^2.
\end{aligned}
\end{align}

\begin{lemma}\label{lemma:auxiliary_identity1}
We assume that the correction factors in \eqref{semi_fem_u_2D_afc_linear} are defined through Definition \ref{definition:correction_factors_AFC}. Then, 
\begin{equation}\label{error_eq_4}
\begin{aligned}
& \int_0^t (\II^{1-\alpha}\widehat \theta(\tau), \widehat \theta(\tau))_h\,d\tau + \tribar\dwidehat \theta(t)\tribar^2 \\
& \qquad \leq C_{\alpha,\mu}\left( \int_0^t (\II^{1-\alpha}\widehat \rho(\tau), \widehat \rho(\tau))_h\,d\tau + t^{2}\int_0^t\|\theta(\tau)\|_{L^2}^2\,d\tau + h^4\,t^{3-\alpha(2-\delta)}\|u_0\|_{\dot H^\delta}^2\right).
\end{aligned}
\end{equation}
\end{lemma}
\begin{proof}
We integrate \eqref{error_eq_1} with respect to time, to get
\begin{equation}\label{error_eq_42}
\begin{aligned}
(\II^{1-\alpha}\theta(t), \chi)_h + \calB_h(\widehat \theta(t), \chi) = - (\II^{1-\alpha}\rho(t), \chi)_h + \int_0^t\widehat d_{\D,h}(u_h(\tau);u_h(\tau),\chi)\,d\tau,
\end{aligned}
\end{equation}
where we have used the identity $\II^{2-\alpha}\theta^\prime(t) = \II^{1-\alpha}\theta(t) - \omega_{2-\alpha}(t)\theta(0)$ with $\theta(0)=0.$ Also, the stabilization term on the right hand side, can be written as
\begin{align*}
\int_0^t\widehat d_{\D,h}(u_h(\tau);u_h(\tau),\chi)\,d\tau  = \sum_{i<j}(\chi_i-\chi_j)\int_0^t\varrho_{ij}(u_h(\tau))\,d\tau,
\end{align*}
where for any $\psi\in\Sh,$ we define $\varrho_{ij}(\psi) := d_{ij}(1 - \mathfrak{a}_{ij}(\psi))(\psi_i - \psi_j),\,i,j=1,\ldots,\N.$

After an integration with respect to time, we choose $\chi=\widehat\theta(t)\in\Sh,$ to get
\begin{align}\label{error_eq_3}
(\II^{1-\alpha}\widehat \theta(t), \widehat \theta(t))_h + \calB_h(\dwidehat \theta(t), \widehat \theta(t)) = - (\II^{1-\alpha}\widehat \rho(t), \widehat \theta(t))_h +  \sum_{i<j}(\widehat \theta_i(t)-\widehat \theta_j(t))\int_0^t\int_0^\tau\varrho_{ij}(u_h(s))\,ds\,d\tau.
\end{align}
We integrate the above again we respect to time, to get
\begin{equation}\label{error_eq_32}
\begin{aligned}
& \int_0^t(\II^{1-\alpha}\widehat \theta(\tau), \widehat \theta(\tau))_h\,d\tau + \int_0^t\frac{d}{d\tau}\tribar \dwidehat \theta(\tau) \tribar^2\,d\tau = - \int_0^t (\bfb \cdot \nabla \dwidehat \theta(\tau), \widehat \theta(\tau))\,d\tau\\
& \qquad  -  \int_0^t((\II^{1-\alpha}\widehat \rho(\tau), \widehat \theta(\tau))_h\,d\tau +  \sum_{i<j}\int_0^t(\widehat \theta_i(\tau)-\widehat \theta_j(\tau))\int_0^\tau\int_0^s\varrho_{ij}(u_h(\zeta))\,d\zeta\,ds\,d\tau.
\end{aligned}
\end{equation}
The first term on the right hand side is vanished, due to the fact that $\ddiv \bfb=0$ and the zero Dirichlet boundary conditions. In addition, the last term on the right hand side term, can be written as
\begin{align*}
J(t) & := \sum_{i<j}\int_0^t(\widehat \theta_i(\tau)-\widehat \theta_j(\tau))\int_0^\tau\int_0^s\varrho_{ij}(u_h(\zeta))\,d\zeta\,ds\,d\tau\\
& = \sum_{i<j}\int_0^t\left(\frac{d}{d\tau}(\dwidehat \theta_i(\tau)-\dwidehat \theta_j(\tau))\int_0^\tau(\tau-\zeta)\varrho_{ij}(u_h(\zeta))\,d\zeta\right)d\tau,
\end{align*}
since  $\frac{d}{dt}\dwidehat \theta(t) = \widehat\theta(t).$ 
Thus, using integration by parts, the above term becomes
\begin{align*}
J(t) & = \sum_{i<j}(\dwidehat \theta_i(t) - \dwidehat \theta_j(t))\int_0^t(t-\zeta)\varrho_{ij}(u_h(\zeta))\,d\zeta - \sum_{i<j}\int_0^t(\dwidehat \theta_i(\tau) - \dwidehat \theta_j(\tau))\int_0^\tau\varrho_{ij}(u_h(s))\,ds = J_1(t) + J_2(t).
\end{align*}
Now, we need to estimate the $J_1(t),\,J_2(t).$ We get in view of Lemma \ref{lemma:stability_widehat_D_main}, \eqref{ritz_projection_est2_2D} and \eqref{regularity},
\begin{align*}
J_1(t) & = \int_0^t (t-\zeta)\,\widehat d_{\D,h}(u_h(\zeta);u_h(\zeta),\dwidehat \theta(t))\,d\zeta\\
& \leq Ch\|\nabla \dwidehat \theta(t)\|_{L^2} \int_0^t (t-\zeta)\,(\|\nabla \theta(\zeta)\|_{L^2} + \|\nabla \rho(\zeta)\|_{L^2} + h\zeta^{-\alpha(2-\delta)/2}\|u_0\|_{\dot H^\delta})\,d\zeta\\
& \leq Ch\|\nabla \dwidehat \theta(t)\|_{L^2} \int_0^t (t-\zeta)\,(\|\nabla \theta(\zeta)\|_{L^2} + h\zeta^{-\alpha(2-\delta)/2}\|u_0\|_{\dot H^\delta})\,d\zeta\\
& \leq C_{\alpha,\mu}h\tribar \dwidehat \theta(t)\tribar  \left(t^{3/2}\left(\int_0^t\|\nabla \theta(\tau)\|_{L^2}^2\,d\tau\right)^{1/2} + h\,t^{2-\alpha(2-\delta)/2}\|u_0\|_{\dot H^\delta} \right).
\end{align*}
For the second term, we have also in view of Lemma \ref{lemma:stability_widehat_D_main},
\begin{align*}
J_1^2 & = \int_0^t\int_0^\tau \widehat d_{\D,h}(u_h(s);u_h(s),\dwidehat \theta(\tau))\,ds\,d\tau\\
& \leq Ch\int_0^t\|\nabla \dwidehat \theta(\tau)\|_{L^2}\left( \int_0^\tau(\|\nabla (u_h(s) - u(s))\|_{L^2} + hs^{-\alpha(2-\delta)/2}\|u_0\|_{\dot H^\delta})\,ds\right) d\tau\\
& = C_{\mu}h \left(\int_0^t\tribar \dwidehat \theta(\tau)\tribar\,d\tau\right) \left(\int_0^t(\|\nabla (u_h(\tau) - u(\tau))\|_{L^2} + hs^{-\alpha(2-\delta)/2}\|u_0\|_{\dot H^\delta})\,d\tau \right)\\
& \leq C_{\mu,\alpha}h t^{1/2}\left(\int_0^t\tribar \dwidehat \theta(\tau)\tribar^2\,d\tau\right)^{1/2}  \left(t^{1/2}\left(\int_0^t(\|\nabla \theta(\tau)\|_{L^2}^2\,d\tau\right)^{1/2} + ht^{1-\alpha(2-\delta)/2}\|u_0\|_{\dot H^\delta} \right),
\end{align*}
where we have assume that $1-\alpha(2-\delta)/2 > 0,$ which is true for $\delta\in[0,2]$ and $\alpha\in(0,1).$
Inserting the above estimates to the error equation \eqref{error_eq_1} and using \eqref{derivative_cont}, we get
\begin{align*}
& \int_0^t (\II^{1-\alpha}\widehat \theta(\tau), \widehat \theta(\tau))_h\,d\tau + \tribar\dwidehat \theta(t)\tribar^2 \leq C\int_0^t\tribar\dwidehat \theta(\tau)\tribar^2\,d\tau\\
& \qquad + \int_0^t (\II^{1-\alpha}\widehat \rho(\tau), \widehat \rho(\tau))_h\,d\tau + C_{\alpha,\mu}h^2 t^{2}\int_0^t\|\nabla \theta(\tau)\|_{L^2}^2\,d\tau + C_{\alpha,\mu}h^4\,t^{3-\alpha(2-\delta)}\|u_0\|_{\dot H^\delta}^2.
\end{align*}
Using inverse inequality and the Gronwall's in integral form, we obtain
\begin{align*}
& \int_0^t (\II^{1-\alpha}\widehat \theta(\tau), \widehat \theta(\tau))_h\,d\tau + \tribar\dwidehat \theta(t)\tribar^2 \\
& \qquad \leq  C_{\alpha,\mu}\left( \int_0^t (\II^{1-\alpha}\widehat \rho(\tau), \widehat \rho(\tau))_h\,d\tau + t^{2}\int_0^t\|\theta(\tau)\|_{L^2}^2\,d\tau + h^4\,t^{3-\alpha(2-\delta)}\|u_0\|_{\dot H^\delta}^2\right).
\end{align*}
\end{proof}

\begin{lemma}\label{lemma:auxiliary_identity2}
We assume that the correction factors in \eqref{semi_fem_u_2D_afc_linear} are defined through Definition \ref{definition:correction_factors_AFC}. Then, 
\begin{align*}
& \int_0^t(\II^{1-\alpha}\theta_1(\tau), \theta_1(\tau))_h\,d\tau + \tribar \widehat \theta_1(t)\tribar^2\\
& \quad \leq C_{\alpha,\mu}\left( \int_0^t \left(  (\II^{1-\alpha}\rho_1(\tau), \rho_1(\tau))_h + (\II^{1-\alpha}\widehat \rho(\tau), \widehat \rho(\tau))\right)_h\,d\tau + t^2\int_0^t \|\theta(\tau)\|_{L^2}^2 \,d\tau + h^4\tau^{3-\alpha(2-\delta)}\|u_0\|_{\dot H^\delta}^2 \right).
\end{align*}
\end{lemma}
\begin{proof}
We multiply the error equation \eqref{error_eq_1} with $t$ in both sides, to get
\begin{align*}
(t\II^{1-\alpha}\theta^\prime(t), \chi)_h + \calB_h(\theta_1(t), \chi)  = - (t\II^{1-\alpha}\rho^\prime(t), \chi)_h + t\widehat d_{\D,h}(u_h(t); u_h(t) ,\chi). 
\end{align*}
We note that, in view of \eqref{aux_eq_2}, the latter error equation can be equivalently written as
\begin{align}\label{error_eq_7}
(\II^{1-\alpha}\theta_1^\prime(t), \chi)_h + \calB_h(\theta_1(t), \chi) = (\kappa_\alpha(t), \chi)_h  + t\widehat d_{\D,h}(u_h(t);u_h(t),\chi), 
\end{align}
where
\begin{align*}
\kappa_\alpha(t) := \alpha\II^{1-\alpha}\theta(t)  + \II^{1-\alpha}\rho_1^\prime(t) - \alpha\II^{1-\alpha}\rho(t).
\end{align*}
Then, integrating in time, to get
\begin{align*}
(\II^{1-\alpha}\theta_1(t), \chi)_h + \calB_h(\widehat \theta_1(t), \chi)  = (\widetilde \kappa_\alpha(t), \chi)_h + \int_0^t \tau\widehat d_{\D,h}(u_h(\tau);u_h(\tau),\chi)\,d\tau, 
\end{align*}
where
\begin{align*}
\widetilde \kappa_\alpha(t) := \alpha\II^{1-\alpha}\widehat \theta(t) + \II^{1-\alpha}\rho_1(t) - \alpha\II^{1-\alpha}\widehat \rho(t).
\end{align*}
To estimate the stabilization terms on the right hand side, we work as follows.  For the first stabilization, we define for $\psi\in\Sh,$ the function $\varrho_{ij}(\psi) := d_{ij}(1 - \mathfrak{a}_{ij}(\psi))(\psi_i - \psi_j),\,i,j=1,\ldots,\N,$ and
\begin{equation}\label{error_eq_stab_term1}
\begin{aligned}
\int_0^t \tau \widehat d_{\D,h}(u_h(\tau); u_h(t),\chi) & = \sum_{i<j}(\chi_i-\chi_j)\int_0^t \tau \varrho_{ij}(u_h(\tau))\,d\tau  = \sum_{i<j}\widetilde G_{ij}(t)(\chi_i-\chi_j),
\end{aligned}
\end{equation}
where
\begin{align*}
\widetilde G_{ij}(t) := \int_0^t \tau \varrho_{ij}(u_h(\tau))\,d\tau,\;\;i,j=1,\ldots,\N.
\end{align*}
Then, the error equation can be written for all $\chi\in\Sh,$ as
\begin{align}\label{error_eq_2}
(\II^{1-\alpha}\theta_1(t), \chi)_h + \calB_h(\widehat \theta_1(t), \chi) = (\widetilde \kappa_\alpha(t), \chi)_h + \sum_{i<j}(\chi_i - \chi_j)\widetilde G_{ij}(t). 
\end{align}
Setting $\chi=\theta_1(t)\in\Sh$ and integrating again over time, we get
\begin{align*}
\int_0^t(\II^{1-\alpha}\theta_1(\tau), \theta_1(\tau))_h\,d\tau  & + \int_0^t\calB_h(\widehat \theta_1(\tau), \theta_1(\tau))\,d\tau\\
& = \int_0^t(\widetilde \kappa_\alpha(\tau), \theta_1(\tau))_h\,d\tau + \sum_{i<j}\int_0^t((\theta_1(\tau))_i - (\theta_1(\tau))_j)\widetilde G_{ij}(\tau)\,d\tau.
\end{align*}
For the first term on the right hand side, we obtain using \eqref{derivative_cont}
\begin{align*}
\int_0^t(\widetilde \kappa_\alpha(\tau), \theta_1(\tau))_h\,d\tau & =  \int_0^t (\alpha\II^{1-\alpha}\widehat \theta(\tau) + \II^{1-\alpha}\rho_1(\tau) - \alpha\II^{1-\alpha}\widehat \rho(\tau),\theta_1(\tau))_h\,d\tau\\
& \leq C\int_0^t \left(  (\alpha\II^{1-\alpha}\widehat \theta(\tau), \widehat \theta(\tau))_h + (\II^{1-\alpha}\rho_1(\tau), \rho_1(\tau))_h + (\II^{1-\alpha}\widehat \rho(\tau), \widehat \rho(\tau))_h\right)\,d\tau \\
& + \frac{1}{2} \int_0^t (\II^{1-\alpha}\theta_1(\tau),\theta_1(\tau))_h\,d\tau.
\end{align*}
Next, we note that $\frac{d}{dt}\widehat \theta_1(t) = \theta_1(t)$ and thus
\begin{align*}
\sum_{i<j}\int_0^t((\theta_1(\tau))_i - (\theta_1(\tau))_j)\widetilde G_{ij}(\tau) & = \sum_{i<j}\int_0^t\frac{d}{d\tau}((\widehat\theta_1(\tau))_i - (\widehat\theta_1(\tau))_j)\widetilde G_{ij}(\tau) \,d\tau\\
&  = \sum_{i<j}((\widehat\theta_1(t))_i - (\widehat\theta_1(t))_j)\widetilde G_{ij}(t) - \sum_{i<j}\int_0^t((\widehat\theta_1(\tau))_i - (\widehat\theta_1(\tau))_j)\frac{d}{d\tau} \widetilde  G_{ij}(\tau) \,d\tau\\
& = J_1(t) + J_2(t).
\end{align*}
For the first term, we have in view of Lemma \ref{lemma:stability_widehat_D_main},
\begin{align*}
J_1(t) & = \sum_{i<j}((\widehat\theta_1(t))_i - (\widehat\theta_1(t))_j)\int_0^t \tau \varrho_{ij}(u_h(\tau))\,d\tau  = \int_0^t \tau \, \widehat d_{\D,h}(u_h(\tau);u_h(\tau), \widehat\theta_1(t))\,d\tau\\
& \leq Ch\|\nabla \widehat\theta_1(t)\|_{L^2}\int_0^t \tau \left( \|\nabla \theta(\tau)\|_{L^2} + \|\nabla \rho(\tau)\|_{L^2} + h\tau^{-\alpha(2-\delta)/2}\|u_0\|_{\dot H^\delta} \right)\,d\tau\\
& \leq C_{\alpha,\mu}h\tribar \widehat\theta_1(t)\tribar\left( t^{1/2}\left(\int_0^t  \| \nabla \theta_1(\tau) \|_{L^2}^2 \,d\tau\right)^{1/2} + ht^{2-\alpha(2-\delta)/2}\|u_0\|_{\dot H^\delta} \right) \\
& \leq \frac{1}{2}\tribar \widehat\theta_1(t)\tribar^2 +  C_{\alpha,\mu}\left( t^3\int_0^t \|\theta(\tau)\|_{L^2}^2\,d\tau + h^4t^{4-\alpha(2-\delta)}\|u_0\|_{\dot H^\delta}^2 \right),
\end{align*}
where in the last estimate, we have used the inverse inequality and the $\theta_1 = t\theta.$
For the next term, we use again Lemma \ref{lemma:stability_widehat_D_main},
\begin{align*}
J_2(t) & =  \int_0^t\sum_{i<j}((\widehat\theta_1(\tau))_i - (\widehat\theta_1(\tau))_j) \tau \varrho_{ij}(u_h(\tau))\,d\tau\\
& = \int_0^t \tau \, \widehat d_{\D,h}(u_h(\tau);u_h(\tau), \widehat\theta_1(\tau))\,d\tau\\
&  \leq Ch\int_0^t \tau \|\nabla \widehat\theta_1(\tau)\|_{L^2}\left( \|\nabla \theta(\tau)\|_{L^2} + \|\nabla \rho(\tau)\|_{L^2} + h\tau^{-\alpha(2-\delta)/2}\|u_0\|_{\dot H^\delta}\right)\,d\tau\\
&  \leq Ch\int_0^t  \|\nabla \widehat\theta_1(\tau)\|_{L^2}\left( \|\nabla \theta_1(\tau)\|_{L^2} + h\tau^{1-\alpha(2-\delta)/2}\|u_0\|_{\dot H^\delta} \right)\,d\tau\\
&  \leq C_{\mu} \left( \int_0^t \tribar \widehat\theta_1(\tau) \tribar^2\,d\tau \right)^{1/2}\left(\int_0^t\left( h^2\|\nabla  \theta_1(\tau)\|_{L^2}^2 + h^4\tau^{4-\alpha(2-\delta)}\|u_0\|_{\dot H^\delta}^2 \right)\,d\tau\right)^{1/2}.
\end{align*}
Using the above estimates, the error equation can be estimated as
\begin{equation}\label{error_eq_5}
\begin{aligned}
& \int_0^t(\II^{1-\alpha}\theta_1(\tau), \theta_1(\tau))_h\,d\tau + \tribar \widehat \theta_1(t)\tribar^2 \\
& \qquad = C\int_0^t \left(  (\II^{1-\alpha}\widehat \theta(\tau), \widehat \theta(\tau))_h + (\II^{1-\alpha}\rho_1(\tau), \rho_1(\tau))_h + (\II^{1-\alpha}\widehat \rho(\tau), \widehat \rho(\tau))_h\right)\,d\tau\\
& \qquad + C\int_0^t\tribar \widehat\theta_1(\tau) \tribar^2\,d\tau + C_{\alpha,\mu}\left( t^2\int_0^t \| \theta(\tau)\|_{L^2}^2 \,d\tau + h^4\tau^{4-\alpha(2-\delta)}\|u_0\|_{\dot H^\delta}^2\right).
\end{aligned}
\end{equation}
Using \eqref{error_eq_4}, Lemma \ref{lemma:auxiliary_identity1} into \eqref{error_eq_5} and then using again the Gronwall's in integral form, we obtain
\begin{align*}
& \int_0^t(\II^{1-\alpha}\theta_1(\tau), \theta_1(\tau))_h\,d\tau + \tribar \widehat \theta_1(t)\tribar^2\\
& \quad \leq C_{\alpha,\mu}\left( \int_0^t \left(  (\II^{1-\alpha}\rho_1(\tau), \rho_1(\tau))_h + (\II^{1-\alpha}\widehat \rho(\tau), \widehat \rho(\tau))\right)_h\,d\tau + t^2 \int_0^t \| \theta(\tau)\|_{L^2}^2 \,d\tau + h^4\tau^{3-\alpha(2-\delta)}\|u_0\|_{\dot H^\delta}^2 \right).
\end{align*}
\end{proof}

\begin{theorem}\label{theorem:error_estimates}
Let $u$ the unique solution to \eqref{weak_u} with initial data $u_0\in\dot H^\delta,\,0\leq \delta \leq 2,$ which satisfies the regularity estimates \eqref{regularity}. Also, let $u_h\in\Sh$ the unique solution to \eqref{semi_fem_u_2D_afc}, where the correction factors in \eqref{semi_fem_u_2D_afc_linear} are defined through Definition \ref{definition:correction_factors_AFC}.  Then, for $0< \alpha < \frac{1}{2-\delta},$ there exists a constant $C,$ independent of $h,$ such that
\begin{align*}
\|u(t) - u_h(t)\|_{L^2} \leq C_{\alpha,\mu}h^2t^{-\alpha(2-\delta)/2}\|u_0\|_{\dot H^\delta},\;\;\;0< t \leq T.
\end{align*}
\end{theorem}
\begin{proof}
We test $\chi = \theta_1(t)\in\Sh$ into \eqref{error_eq_7}, to get
\begin{align}\label{error_eq_8}
(\II^{1-\alpha}\theta_1^\prime(t), \theta_1(t))_h + \tribar \theta_1(t)\tribar^2 & =  (\kappa_\alpha(t), \theta_1(t))  + t\widehat d_{\D,h}(u_h(t);u_h(t),\theta_1(t)) = J_1(t) + J_2(t),
\end{align}
where
\begin{align*}
\kappa_\alpha(t) := \alpha\II^{1-\alpha}\theta(t)  + \II^{1-\alpha}\rho_1^\prime(t) - \alpha\II^{1-\alpha}\rho(t).
\end{align*}
First, we estimate the term $J_1.$ To achieve that, we need an estimate for $(\II^{1-\alpha}\theta(t),\theta_1(t))_h,$ and thus, we set $\chi = \theta_1(t)\in\Sh$ into \eqref{error_eq_42}, to get
\begin{align*}
(\II^{1-\alpha}\theta(t), \theta_1(t))_h  = - \calB_h(\widehat \theta(t), \theta_1(t)) - (\II^{1-\alpha}\rho(t), \theta_1(t))_h + \int_0^t\widehat d_{\D,h}(u_h(\tau);u_h(\tau),\theta_1(t))\,d\tau.
\end{align*}
To estimate the stabilization term, we use Lemma \ref{lemma:stability_widehat_D_main}, the estimates \eqref{ritz_projection_est2_2D} and the inverse inequality \eqref{eq:inverse_estimate}, 
\begin{align*}
\int_0^t\widehat d_{\D,h}(u_h(\tau);u_h(\tau),\theta_1(t))\,d\tau & \leq Ch\|\nabla \theta_1(t)\|_{L^2}\int_0^t \left(\|\nabla\theta(\tau)\|_{L^2} + h\tau^{-\alpha(2-\delta)/2}\|u_0\|_{\dot H^\delta} \right)\,d\tau\\
& \leq C\tribar \theta_1(t)\tribar \left( t^{1/2}\left(\int_0^t \|\theta(\tau)\|_{L^2}^2\,d\tau\right)^{1/2} + h^2t^{1-\alpha(2-\delta)/2} \|u_0\|_{\dot H^\delta}\right).
\end{align*}
Thus, inserting the latter estimate, we get
\begin{align*}
(\II^{1-\alpha}\theta(t), \theta_1(t))_h & \leq C \left( \tribar \widehat \theta(t)\tribar^2 + t\int_0^t \|\theta(\tau)\|_{L^2}^2\,d\tau + h^4t^{2-\alpha(2-\delta)} \|u_0\|_{\dot H^\delta}^2\right)\\
& + \frac{1}{4}\tribar \theta_1(t)\tribar^2  + C\|\theta_1(t)\|_{L^2}\|\II^{1-\alpha}\rho(t)\|_{L^2}.
\end{align*}
Recall the estimate from Lemma \ref{lemma:auxiliary_identity2} and \eqref{rho_integral_estimate4}, then
\begin{align*}
& \tribar \widehat \theta_1(t) \tribar^2  \leq C_{\alpha,\mu}\left( t^{2}\int_0^t\|\theta(\tau)\|_{L^2}^2\,d\tau + h^4\,t^{3- \alpha(2-\delta)}\|u_0\|_{\dot H^\delta}^2 \right).
\end{align*}
Then, in view of the latter two estimates and \eqref{mass_lump_equivalence}, the term $J_1,$ can be estimated as 
\begin{align}\label{error_eq_15}
\begin{aligned}
& J_1(t)  = (\alpha\II^{1-\alpha}\theta(t) + \II^{1-\alpha}\rho_1^\prime(t) - \alpha\II^{1-\alpha}\rho(t),\theta_1(t))_h \leq \frac{1}{4}\tribar \theta_1(t)\tribar^2 \\
& \quad + C_\alpha \left( t\int_0^t \|\theta(\tau)\|_{L^2}^2\,d\tau + h^4t^{2 - \alpha(2-\delta)} \|u_0\|_{\dot H^\delta}^2 + \|\theta_1(t)\|_{L^2}(\|\II^{1-\alpha}\rho(t)\|_{L^2} + \|\II^{1-\alpha}\rho_1^\prime(t)\|_{L^2}) \right).
\end{aligned}
\end{align}
For the second term $J_2(t)$ of \eqref{error_eq_8}, in view of Lemma \ref{lemma:stability_widehat_D_main} and \eqref{ritz_projection_est2_2D},
\begin{align}\label{error_eq_16}
\begin{aligned}
J_2(t) & \leq Cht\left(\|\nabla \theta(t)\|_{L^2} + \|\nabla \rho(t)\|_{L^2} + ht^{-\alpha(2-\delta)/2}\|u_0\|_{\dot H^\delta}\right)\|\nabla \theta_1(t)\|_{L^2}\\
& \leq C_{\mu}h^2\left(\tribar \theta_1(t)\tribar^2 + h^2t^{2-\alpha(2-\delta)}\|u_0\|_{\dot H^\delta}^2\right) + \frac{1}{4}\tribar \theta_1(t)\tribar^2.
\end{aligned}
\end{align}
Therefore, inserting \eqref{error_eq_15} and \eqref{error_eq_16} into \eqref{error_eq_8}, and for sufficiently small $h$ such that $1/2 - C_\mu h^2>0,$ we get
\begin{align*}
(\II^{1-\alpha}\theta_1^\prime(t), \theta_1(t))_h + \tribar \theta_1(t)\tribar^2 & \leq  C_{\alpha,\mu}\left( t\int_0^t \|\theta(\tau)\|_{L^2}^2\,d\tau + h^4t^{2 -\alpha(2-\delta)} \|u_0\|_{\dot H^\delta}^2\right)\\
& + C_\mu\|\theta_1(t)\|_{L^2}\left(\|\II^{1-\alpha}\rho(t)\|_{L^2} + \|\II^{1-\alpha}\rho_1^\prime(t)\|_{L^2}\right).
\end{align*}
In view of Lemma \ref{lemma:derivative_caputo_inequality}, the latter becomes
\begin{align*}
\partial_t^\alpha \|\theta_1(t)\|_h^2 +  \tribar \theta_1(t)\tribar^2 & \leq C_{\alpha,\mu}\left( t\int_0^t \|\theta(\tau)\|_{L^2}^2\,d\tau + h^4t^{2 -\alpha(2-\delta)} \|u_0\|_{\dot H^\delta}^2\right)\\
& + C_\mu\|\theta_1(t)\|_{L^2}\left(\|\II^{1-\alpha}\rho(t)\|_{L^2} + \|\II^{1-\alpha}\rho_1^\prime(t)\|_{L^2}\right).
\end{align*}
We operate both sides with $\partial_t^{-\alpha}$ and using \eqref{caputo_derivative} and \eqref{mass_lump_equivalence}, we get
\begin{align}\label{error_eq_18}
\begin{aligned}
\|\theta_1(t)\|_{L^2}^2 +  \partial_t^{-\alpha}\tribar \theta_1(t)\tribar^2 & \leq C_{\alpha,\mu}\left( \int_0^t (t-\tau)^{\alpha-1}\tau y(\tau)\,d\tau + h^4\|u_0\|_{\dot H^\delta}^2\int_0^t(t-\tau)^{\alpha-1}\tau^{2 -\alpha(2-\delta)}\,d\tau\right)\\
& + C_\alpha\int_0^t(t-\tau)^{\alpha-1}\|\theta_1(\tau)\|_{L^2}\left(\|\II^{1-\alpha}\rho(\tau)\|_{L^2} + \|\II^{1-\alpha}\rho_1^\prime(\tau)\|_{L^2}\right)\,d\tau,
\end{aligned}
\end{align}
with 
\begin{align*}
y(t):=\int_0^t \|\theta(\tau)\|_{L^2}^2\,d\tau.
\end{align*}
The three integrals on the right hand side of \eqref{error_eq_18}, can be estimated as follows
\begin{align*}
\int_0^t(t-\tau)^{\alpha-1}\tau y(\tau)\,d\tau \leq  y(t)\int_0^t(t-\tau)^{\alpha-1}\tau\,d\tau = C_\alpha y(t)t^{1+\alpha},
\end{align*}
since $y^\prime(t)\geq 0,\,t\in [0,T].$ Further,  
\begin{align*}
\int_0^t(t-\tau)^{\alpha-1}\tau^{2 -\alpha(2-\delta)}\,d\tau = C_\alpha t^{2 + \alpha -\alpha(2-\delta)}.
\end{align*}
For the last integral of \eqref{error_eq_18}, we obtain from \eqref{rho_integral_estimate},
\begin{align*}
& \int_0^t(t-\tau)^{\alpha-1}\|\theta_1(\tau)\|_{L^2}\left(\|\II^{1-\alpha}\rho(\tau)\|_{L^2} + \|\II^{1-\alpha}\rho_1^\prime(\tau)\|_{L^2}\right)\,d\tau\\
&\qquad \leq \max_{0\leq \tau\leq t}\|\theta_1(\tau)\|_{L^2}\int_0^t(t-\tau)^{\alpha-1}\left(\|\II^{1-\alpha}\rho(\tau)\|_{L^2} + \|\II^{1-\alpha}\rho_1^\prime(\tau)\|_{L^2}\right)\,d\tau\\
& \qquad \leq C_\alpha h^2\|u_0\|_{\dot H^\delta}\max_{0\leq \tau\leq t}\|\theta_1(\tau)\|_{L^2}\int_0^t(t-\tau)^{\alpha-1}\left(\tau^{1-\alpha-\alpha(2-\delta)/2}\right)\,d\tau\\
& \qquad = C_\alpha h^2 t^{1-\alpha(2-\delta)/2}\|u_0\|_{\dot H^\delta}\max_{0\leq \tau\leq t}\|\theta_1(\tau)\|_{L^2}.
\end{align*}
Due to the continuity of $\|\theta_1(t)\|_{L^2}$ on $[0,t],$ there exists a $t^*\in [0,t]$ such that
\begin{align*}
\max_{0\leq \tau\leq t}\|\theta_1(\tau)\|_{L^2} = \|\theta_1(t^*)\|_{L^2}.
\end{align*}
Then,
\begin{align*}
\|\theta_1(t^*)\|_{L^2}^2 & \leq C_{\alpha,\mu}\left( t^{1+\alpha}\int_0^t \|\theta(\tau)\|_{L^2}^2\,d\tau + h^4t^{2+\alpha-\alpha(2-\delta)}\|u_0\|_{\dot H^\delta}^2 + h^2\|\theta_1(t^*)\|_{L^2} t^{1-\alpha(2-\delta)/2}\|u_0\|_{\dot H^\delta}\right).
\end{align*}
Then, using the Young's inequality, we get for $0\leq t\leq T,$ 
\begin{align*}
\|\theta_1(t^*)\|_{L^2}^2 \leq C_{\alpha,\mu}\left( t^{1+\alpha}\int_0^t \|\theta(\tau)\|_{L^2}^2\,d\tau + h^4t^{2-\alpha(2-\delta)}\|u_0\|_{\dot H^\delta}^2\right).
\end{align*}
Further, since $\theta_1(t)=t\theta(t),$ we obtain for $0< t\leq T,$ 
\begin{align*}
\|\theta(t)\|_{L^2}^2 \leq C_{\alpha,\mu}\left( t^{\alpha-1}\int_0^t \|\theta(\tau)\|_{L^2}^2\,d\tau + h^4t^{-\alpha(2-\delta)}\|u_0\|_{\dot H^\delta}^2\right).
\end{align*}
Further, according to the above notation and letting $g(t):=h^4t^{-\alpha(2-\delta)}\|u_0\|_{\dot H^\delta}^2$, we have 
\begin{align}\label{error_eq_17}
y'(t) \leq C_{\alpha,\mu}\left( t^{\alpha-1}y(t) + g(t)\right).
\end{align}
We note that $\II^1(g(t))$ is finite only when $0< \alpha < \frac{1}{2-\delta},$ and $\II^1(t^{\alpha-1})$ is finite for all $\alpha\in(0,1),$ thus, we can apply the Gronwall's inequality in differential form which yields
\begin{align*}
y(t) \leq C_{\alpha,\mu,T}h^4t^{1-\alpha(2-\delta)}\|u_0\|_{\dot H^\delta}^2.
\end{align*}
Then, using the latter into \eqref{error_eq_17}, yields for $0<\alpha<\frac{1}{2-\delta}$ and $0<t\leq T,$
\begin{align*}
\|\theta(t)\|_{L^2} \leq C_{\alpha,\mu}h^2t^{-\alpha(2-\delta)/2}\|u_0\|_{\dot H^\delta}.
\end{align*}
\end{proof}

\section{A fully--discrete scheme that satisfies the DMP}\label{section:fully}

For the temporal discretization, we will use the L1 method, on uniform meshes. Particularly, let $\NT\ge1$ be a positive integer,
and the time step $k=T/\NT$ so that $0=t^0<t^1<\cdots<t^{\NT}=T$ and $t^n=nk$, $n=0,\dots, \NT$. For $n\geq 1,$ the L1 approximation of the Caputo fractional derivative $\partial_t^\alpha u(\x,t^n),$ can be represented, see, e.g., \cite[Section 3]{lin2007}, \cite[Chapter 5]{jin2023}, and the references therein, as
\begin{align}\label{caputo_derivative_full}
\partial_t^\alpha u(\x,t^n) =  d^0 u(\x,t^n) - d^{n-1}u(\x,t^0) + \sum_{j=1}^{n-1}(d^j - d^{j-1})u(\x,t^{n-j}),
\end{align}
where the weights $d^{j},$ are given by
\begin{align}\label{d_weights}
d^{j} := k^{-\alpha}\frac{(j+1)^{1-\alpha} - j^{1-\alpha}}{\Gamma(2-\alpha)},\;\;\;j=0,\ldots,\NT-1.
\end{align}
We note that $d^{\ell+1} < d^{\ell}$ for $\ell=0,\ldots,\NT-1.$
\begin{remark}
In \cite[Equation (3.3)]{lin2007}, also in \cite[Lemma 4.1]{sun2006}, was proved that the local truncation error of the L1 approximation is bounded by $c_1 k^{2-\alpha},$ for some constant $c_1$ depending only on $u$ and the solution $u$ is sufficient regular, in particular twice continuously differentiable.
\end{remark}
We seek $U^n\in\Sh$ approximation of $u^n:=u(\cdot,t^n)$ for $n=1,\ldots,\NT,$ such that for all $\chi\in\Sh,$ 
\begin{align}\label{L1_fully_discrete}
(\partial_t^\alpha U^n, \chi)_h + \calB_h(U^n,\chi) - \widehat d_{\D,h}(U^n;U^n,\chi) = ( G^n, \chi).
\end{align}
Applying the \eqref{caputo_derivative_full} into \eqref{L1_fully_discrete}, we get
\begin{equation}\label{L1_fully_discrete_expanded}
\begin{aligned}
d^0(U^n, \chi)_h & + \calB_h(U^n,\chi) - \widehat d_{\D,h}(U^n;U^n,\chi) \\
& = d^{n-1}(U^0, \chi)_h - \sum_{\ell=1}^{n-1}(d^{\ell} - d^{\ell-1})(U^{n-\ell},\chi)_h + ( G^n, \chi).
\end{aligned}
\end{equation}
Since the fully--discrete scheme \eqref{L1_fully_discrete} is nonlinear due to the presence of the nonlinear contributions of the stabilization terms, we employ a similar argument to Theorem \ref{theorem:fixed_point_semi_discrete} in order to prove its existence and uniqueness.

\begin{theorem}[Well--posedness of fully--discrete scheme]
We assume that the correction factors in \eqref{semi_fem_u_2D_afc} are defined through \eqref{led_2D_weakened} and Algorithm \ref{algorithm-1}. Then, the fully--discrete scheme \eqref{L1_fully_discrete} has a unique solution.
\end{theorem}
\begin{proof}
We assume that the nonlinear fully--discrete scheme \eqref{L1_fully_discrete_expanded} has a unique solution $U^{\ell}\in\Sh$ for the time steps $t^{\ell},\,\ell=0,\ldots,n-1.$ In the spirit of Theorem \ref{theorem:fixed_point_semi_discrete}, to show the existence of the solution at time step $t=t^{n}$, we introduce the linear operator $\Gh\,:\,\Sh\to \Sh$ that $\Sh\ni v\mapsto \Gh v\in\Sh,$ by
\begin{equation}\label{L1_fully_discrete_linearized}
\begin{aligned}
d^0(\Gh v, \chi)_h & + \calB_h(\Gh v,\chi) - d_{\D,h}(\Gh v,\chi) + \widetilde d_{\D,h}(v;v,\chi) \\
& = d^{n-1}(U^0, \chi)_h - \sum_{\ell=1}^{n-1}(d^{\ell} - d^{\ell-1})(U^{n-\ell},\chi)_h + (G^n, \chi).
\end{aligned}
\end{equation}
Next, we show that the linear operator $\Gh$ is contractive on $\Sh$ for small $h.$ Indeed, let $\Gh v_1,\,\Gh v_2\in\Sh,$ the solution of \eqref{L1_fully_discrete_linearized} with $v_1,\,v_2\in\Sh,$ respectively. We also set $v := v_1 - v_2\in\Sh$ and $\Gh v : = \Gh v_1 - \Gh v_2\in\Sh,$ then in view of \eqref{L1_fully_discrete_linearized}, we get
\begin{align}\label{existence_ineq_fully_discrete}
d^{0}(\Gh v, \chi)_h & + \tribar \Gh v \tribar^2 - d_{\D,h}(\Gh v,\Gh v) = \widetilde d_{\D,h}(v_2; v_2,\chi) - \widetilde d_{\D,h}(v_1; v_1,\chi),\;\;\forall\chi\in \Sh.
\end{align}
As in Theorem \ref{theorem:fixed_point_semi_discrete}, we set $\chi = \Gh v$ into \eqref{existence_ineq_fully_discrete} and use the Lemma \ref{lemma:stability_widehat_D_difference} and similar arguments to conclude to
\begin{align*}
\tribar \Gh v\tribar^2 & \leq C\mu h\tribar v\tribar^2,
\end{align*}
from where choosing $h$ such that $C\mu h<1,$ the operator $\Gh\,:\,\Sh\to \Sh$ is contractive and as a result, there exist a unique solution to \eqref{L1_fully_discrete_expanded} at the time step $t=t^n.$
\end{proof}
Let $\be^n=(\beta^n_1,\dots,\beta^n_{\N})^T$ be the coefficient vector, with respect to the basis of $\Sh$ of $U^n\in\Sh.$ Then, \eqref{L1_fully_discrete_expanded} can be written in the following form,
\begin{equation}\label{L1_fully_discrete_matrix}
\begin{aligned}
& \left(d^0\M_L +  (\mu\S + \Q + \D + \sigma\,\M_L)\right) \be^n\\
& \qquad\qquad = d^{n-1} \M_L\be^{0} + \bfr(G^n) - \sum_{\ell=1}^{n-1}(d^{\ell} - d^{\ell-1})\M_L\be^{n-\ell} +  \overline{\mathsf{f}}(\be^n).
\end{aligned}
\end{equation}
In the following two theorems, we will prove that the stabilized fully--discrete scheme satisfies the local and global DMP, respectively.

\begin{definition}(Local DMP)\label{definition:DMP_fully}
We say that the solution $\be^n \in \mathbb{R}^{\N}$ of \eqref{L1_fully_discrete_matrix} with $\sigma>0$ in $\omega_i$ at $t=t^n,$ satisfies the local DMP if for $i\in\Nh^0,$ we have
\begin{align}
G^n & \leq 0 \;\;\text{in}\;\;\omega_i  \Rightarrow \;\beta_i^n \leq \beta_i^{\max,+},\label{local_dmp1}\\
G^n & \geq 0  \;\;\text{in}\;\;\omega_i  \Rightarrow \;\beta_i^n \geq \beta_i^{\min,-},\label{local_dmp2}
\end{align}
where
\begin{align*}
\beta_i^{\max,+} &:= \max\left\{0, \beta_i^{\max}\right\},\;\;\text{with}\;\;\beta_i^{\max} := \max\left\{\max_{0\leq \ell \leq n-1}\max_{j\in\Zh^i\cup\{i\}}\beta_j^{\ell}, \max_{j\in\Zh^i} \beta_j^n\right\}\\
\beta_i^{\min,-} &:= \min\left\{0, \beta_i^{\min}\right\},\;\;\;\text{with}\;\;\beta_i^{\min} := \min\left\{\min_{0\leq \ell \leq n-1}\min_{j\in\Zh^i\cup\{i\}}\beta_j^{\ell}, \min_{j\in\Zh^i} \beta_j^n\right\}.
\end{align*}
If $\sigma = 0$ in $\omega_i,$ then,
\begin{align}
G^n & \leq 0 \;\;\text{in}\;\;\omega_i  \Rightarrow \;\beta_i^n \leq \beta_i^{\max},\label{local_dmp1_2}\\
G^n & \geq 0  \;\;\text{in}\;\;\omega_i  \Rightarrow \;\beta_i^n \geq \beta_i^{\min}.\label{local_dmp2_2}
\end{align}
\end{definition}

\begin{theorem}\label{theorem:dmp_fully}
Assume the correction factors are computed as in Definition \ref{definition:correction_factors_AFC}. Then, the solution $U^n\in\Sh$ of the fully--discrete scheme \eqref{L1_fully_discrete_expanded} satisfies the local DMP for all $k,\,h,$ in the sense of Definition \ref{definition:DMP_fully}.
\end{theorem}
\begin{proof}
We only prove \eqref{local_dmp1}, are the remaining cases can be treated analogously. Let $i\in\Nh^0,$ we assume that $G^n \leq 0$ in $\omega_i.$ Then, $r_i(G^n)\leq 0,$ since the underling basis functions are positive. The $i-$th equation of \eqref{L1_fully_discrete_matrix} can be written as
\begin{align*}
& \left( d^0m_i + (\mu s_{ii} + \tau_{ii} + d_{ii} + \sigma m_i)\right)\beta_i^n  + \sum_{j\in\Zh^i}(\mu s_{ij} + \tau_{ij} + d_{ij})\beta_j^n\\
& \qquad = d^{n-1}m_i\beta_i^0 - \sum_{\ell=1}^{n-1}(d^{\ell}-d^{\ell-1})m_i\beta_j^{n-\ell} + k\sum_{j\in \Zh^i}d_{ij}\mathfrak{a}_{ij}^n (\beta_j^n - \beta_i^n) + r_i(G^n).
\end{align*}
It is sufficient to assume that $\beta_i^n > \beta_i^{\max,+},$ otherwise \eqref{local_dmp1} holds trivially. Then, at node $i$ the solution attains the local maximum, i.e., $\beta_i^n \geq  \beta_j^n,\,\forall\,j\in \Zh^i.$ Then, $\mathfrak{a}_{ij}^n = 0$ by construction, see Algorithm \ref{algorithm-1}. Also, we note that the matrix $\mu\S + \Q + \D$ has non--positive off--diagonal elements, and therefore,
\begin{align*}
\sum_{j\in\Zh^i}(\mu s_{ij} + \tau_{ij} + d_{ij})\beta_j^n & \geq \beta_i^n\sum_{j\in\Zh^i}(\mu s_{ij} + \tau_{ij} + d_{ij}),
\end{align*}
and since
\begin{align}\label{dmp_ineq_1}
(\mu s_{ii} + \tau_{ii} + d_{ii}) + \sum_{j\in \Zh^i\cap \Nh^0}(\mu s_{ij} + \tau_{ij} + d_{ij}) > 0,
\end{align}
we obtain that 
\begin{align}\label{dmp_ineq_2}
\left( d^0 + \sigma \right)m_i \beta_i^n \leq d^{n-1}m_i\beta_i^0 - \sum_{\ell=1}^{n-1}(d^{\ell+1}-d^{\ell})m_i\beta_j^{n-\ell}
\end{align}
To show \eqref{dmp_ineq_1}, we recall that for the matrices it holds
\begin{align*}
\sum_{j\in \Nh}s_{ij} = \sum_{j\in \Nh}\tau_{ij} = \sum_{j\in \Nh}d_{ij} = 0,\;\;\;\forall\,i\in \Nh^0.
\end{align*}
Then, since all they have non--positive off--diagonal elements, we obtain  for all $i\in \Nh^0,$
\begin{align*}
0 = \sum_{j\in \Nh} (\mu s_{ij} + \tau_{ij} + d_{ij}) & = \sum_{j\in \Nh^0} (\mu s_{ij} + \tau_{ij} + d_{ij}) + \sum_{j\in \Nh^\partial} (\mu s_{ij} + \tau_{ij} + d_{ij})\\
& \leq \sum_{j\in \Nh^0} (\mu s_{ij} + \tau_{ij} + d_{ij}),
\end{align*}
where in the last estimate we have used the non--positivity of the off--diagonal elements of the matrix $\mu\S + \Q + \D.$
Returning to \eqref{dmp_ineq_2}, we use our hypothesis, i.e., $\beta_i^n > \beta_i^{\max,+}$ and the negativity of $d^{\ell}-d^{\ell-1},$ to get
\begin{align*}
\left(d^0 + \sigma \right)m_i \beta_i^n & \leq \left(d^{n-1} - \sum_{\ell=1}^{n-1}(d^{\ell}-d^{\ell-1})\right)m_i\beta_i^{\max,+} = d^0m_i\beta_i^{\max,+}.
\end{align*}
Since $\sigma>0,$ we obtain that $\beta_i^n \leq \beta_i^{\max,+}.$
\end{proof}

\begin{definition}(Global DMP)\label{definition:DMP_fully_global}
We say that the solution $\be^n \in \mathbb{R}^{\N}$ of \eqref{L1_fully_discrete_matrix} with $\sigma>0$ in $\Omega$ at $t=t^n,$ satisfies the global DMP if for $i\in\Nh^0,$ we have
\begin{align}
G^n & \leq 0 \;\;\text{in}\;\;\Omega  \Rightarrow \;\beta_i^n \leq \beta_i^{\max,n-1,+},\label{global_dmp1}\\
G^n & \geq 0  \;\;\text{in}\;\;\Omega  \Rightarrow \;\beta_i^n \geq \beta_i^{\min,n-1,-},\label{global_dmp2}
\end{align}
where 
\begin{align*}
\beta_i^{\max,n-1,+} &:= \max\left\{0, \beta_i^{\max,n-1}\right\},\;\;\text{with}\;\;\beta_i^{\max,n-1} := \max_{j\in \Nh}\beta_j^{n-1}\\
\beta_i^{\min,n-1,-} &:= \min\left\{0, \beta_i^{\min,n-1}\right\},\;\,\;\;\text{with}\;\;\beta_i^{\min,n-1} := \min_{j\in \Nh}\beta_j^{n-1}.
\end{align*}
If $\sigma = 0$ in $\Omega,$ then,
\begin{align}
G^n & \leq 0  \;\;\text{in}\;\;\Omega  \Rightarrow \;\beta_i^n \leq \beta_i^{\max,n},\label{global_dmp1_2}\\
G^n & \geq 0  \;\;\text{in}\;\;\Omega  \Rightarrow \;\beta_i^n \geq \beta_i^{\min,n}.\label{global_dmp2_2}
\end{align}
\end{definition}

\begin{theorem}\label{theorem:dmp_fully_global}
Assume the correction factors are computed as in Definition \ref{definition:correction_factors_AFC}. Then, the solution $U^n\in\Sh$ of the fully--discrete scheme \eqref{L1_fully_discrete_expanded} satisfies the global DMP for all $k,\,h,$ in the sense of Definition \ref{definition:DMP_fully}.
\end{theorem}
\begin{proof}
The proof uses the same arguments to the proof of Theorem \ref{theorem:dmp_fully}, thus we omit it.
\end{proof}

\section{Numerical experiments}\label{section:numerical_results}
In this section we present several numerical experiments, that will validate our theoretical results. In all the numerical examples we have the following setting of the discretization parameters as well as the correction factors.

As the initial triangulation, we consider a uniform subdivision of $\Omega$, where each side of $\Omega$ is partitioned into $M$ intervals of length $h_0=1/M$, with $M\in\mathbb{N}$. The triangulation $\Th$ is then obtained by splitting each small square into two triangles along one of its diagonals and consists of $2M^2$ right-angled triangles with diameter $h=\sqrt{2}h_0$.

\begin{remark}\label{remark:triangulation}
We note that $\Th$ satisfies Assumption \ref{mesh-assumption}, and therefore, in view of Remark \ref{remark:s_elements}, the corresponding stiffness matrix $\S$ has non-positive off-diagonal entries and positive diagonal entries.
\end{remark}

At each time level $t^n,$ we solve the fixed point iterative scheme \eqref{L1_fully_discrete_linearized} using $U^0:=P_hu_0.$ More precisely, let $\widehat \be=(\widehat \beta_1,\dots,\widehat \beta_{\N})^T$ be the coefficient vector, with respect to the basis of $\Sh$ of $\Gh v\in\Sh.$ Then, \eqref{L1_fully_discrete_linearized} can be written in the following form,
\begin{equation}\label{L1_fully_discrete_matrix_linear}
\begin{aligned}
& \left(d^0\M_L +  (\mu\S + \Q + \D + \sigma\,\M_L)\right) \widehat \be\\
& \qquad\qquad = d^{n-1} \M_L\be^{0} + \bfr(G^n) - \sum_{\ell=1}^{n-1}(d^{\ell} - d^{\ell-1})\M_L\be^{n-\ell} +  \overline{\mathsf{f}}(\widetilde \be),
\end{aligned}
\end{equation}
with $\widetilde \be=(\widetilde \beta_1,\dots,\widetilde \beta_{\N})^T$ be the coefficient vector, with respect to the basis of $\Sh$ of $v\in\Sh.$ Also, $\be^0 = \boldsymbol{\xi}^0,$ with $\boldsymbol{\xi}^0\in\mathbb{R}^{\N}$ the coefficient vector of $P_hu_0.$ In all computations of the above linear system, we initialize the Aitken’s procedure. The fixed point iteration is terminated when the relative change between two consecutive iterates is below $\mathsf{TOL} = 10^{-5}$, with a maximum of 50 iterations. For the computation of the correction factors $\mathfrak{a}_{ij},$ we use Algorithm \ref{algorithm-1} with $q_{i}$ as described in Remark \ref{remark:linearity_preservation}, where due to the construction of $\Th$, the $\gamma_i = 1,\; i=1,\ldots,\N.$

To compute the errors, in the following numerical examples, initially, we assume a triangulation $\Th^{(0)}$ with $h_0^{(0)} = 1/M_0,\,M_0=10$ and time step $k^{(0)} = \frac{1}{10}h_0^{(0)}.$ Then, we perform uniform refinement, i.e., on each triangle we connect the midpoints to obtain four similar triangles. Thus, we compute the errors in the triangulations $\Th^{(\ell)},\,\ell=1,\ldots,L,$ with corresponding time step $k^{(\ell)} = \frac{1}{10}h_0^{(\ell)},\,\ell=1,\ldots,L,$ with some $L\geq 2$ positive integer.

\subsection{A numerical example with nonsmooth initial data}

In the first numerical example, we consider the homogeneous time fractional convection--diffusion--reaction equation on  $\Omega = (0,1)^2,$ where $u_0\in \dot H^{1/2-\epsilon},\,\epsilon>0.$ In particular, we recall the initial function $u_0$ as in \cite[Section 4]{karaa2020},
\begin{equation}\label{non_linear_system_f_sol_boundary_layer4}
\begin{aligned}
u_0 = u(x,y,0)  = \mathbf{1}_{D}(x,y),
\end{aligned}
\end{equation}
where $\mathbf{1}_D(x,y)$ denotes the characteristic operator on a domain $D := \{(x,y)\in \Omega\,:\, x^2 + y^2  \leq 1\}.$
Further, we choose the final time $T=10^{-2},$ and $\mu=10^{-12},\,\sigma=0,\,\bfb = (-20,-20)^T,\,\alpha=0.2,\,G=0.$ Then, we compute the profile of the solution for both standard FEM and AFC scheme, using a triangulation $\Th$ with $h_0 = 1/40$ and the time step is set to be $k = 0.01/100.$ The surface plots are presented on Figure \ref{fig:afc_n40} where we can easily see that the AFC scheme produces a non-oscillatory solution whereas the standard FEM is dominated by spurious oscillations.

\begin{figure}
  \centering
  \includegraphics[trim=180 0 180 0, clip, width=0.5\textwidth]{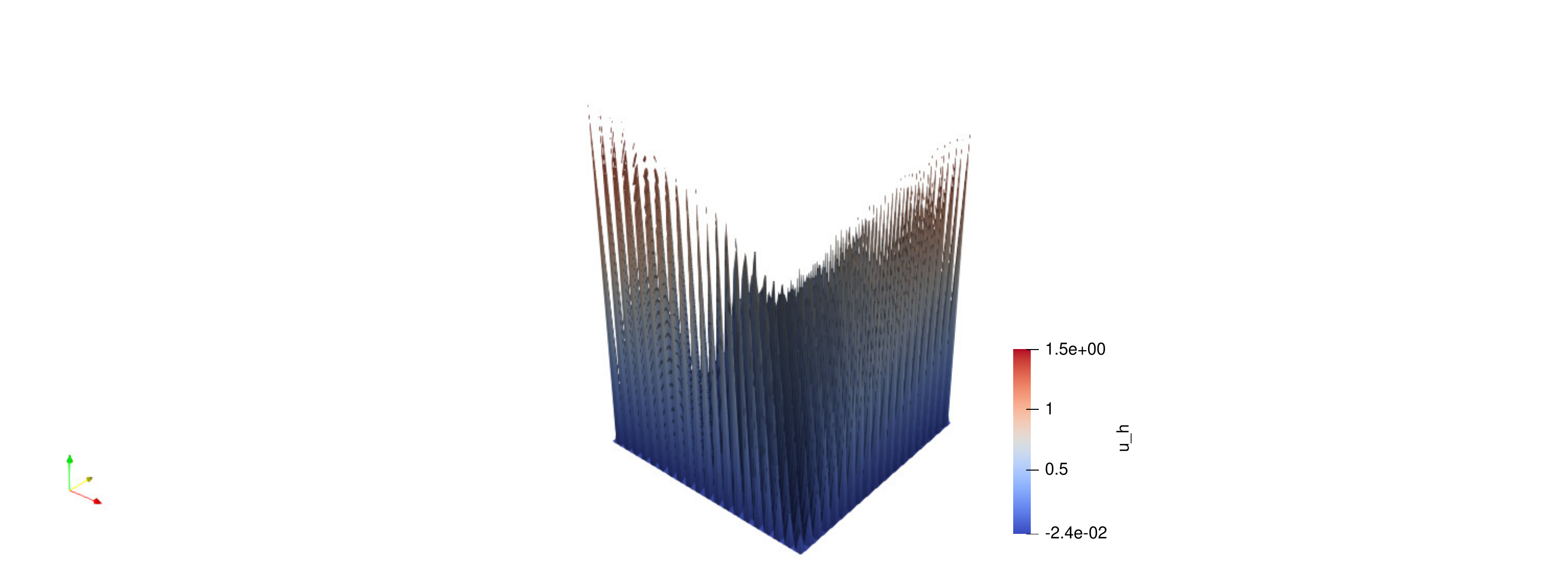}%
  \hfill
  \includegraphics[trim=180 0 180 0, clip, width=0.5\textwidth]{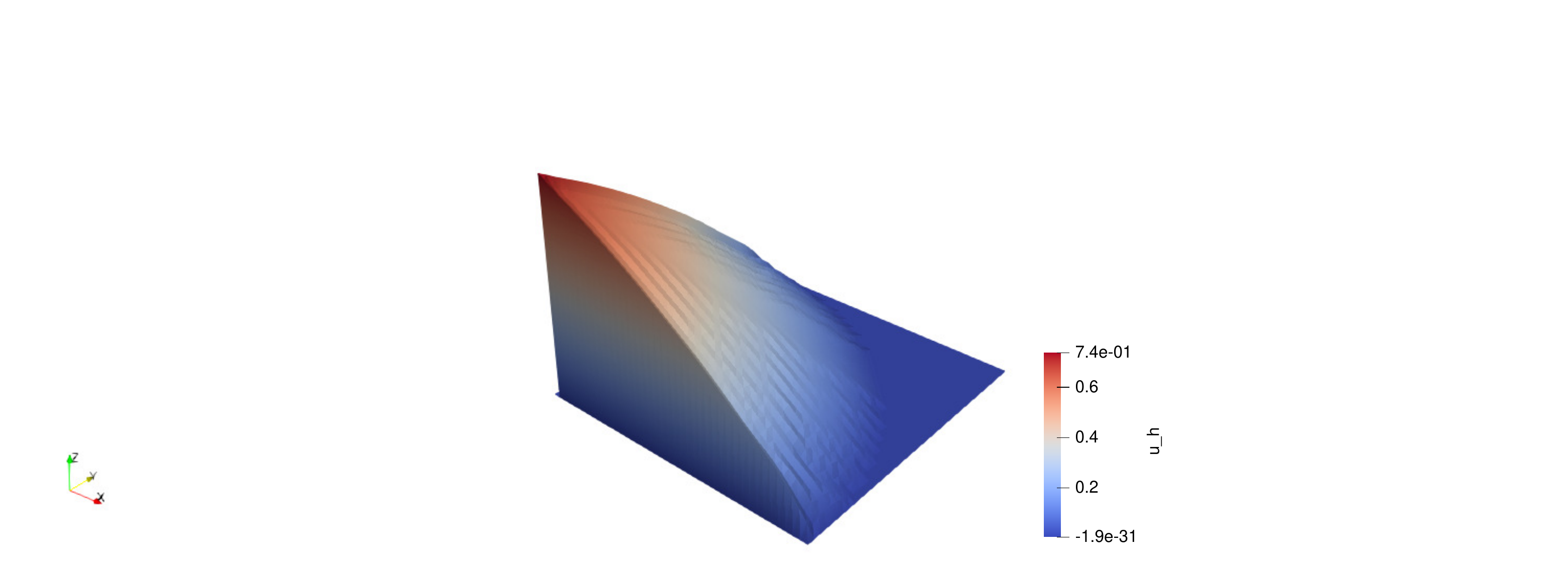}
  \caption{Profile of the solutions of the standard FEM (left) and the AFC scheme (right) for $M=40$ and $\NT=100.$}
  \label{fig:afc_n40}
\end{figure}

\subsection{Convergence study with solution with steep gradients}
We consider the following time fractional convection--diffusion--reaction equation on  $\Omega = (0,1)^2,$ where we choose $T=1,$ and for small values of $\mu,$ we compute the corresponding function $G,$ so as to have the 
\begin{equation}\label{non_linear_system_f_sol_boundary_layer}
\begin{aligned}
u(x,y,t)  = t^2 x (1 - x)  y  (1 - y)\left( \frac{1}{2} + \frac{1}{\pi}  \arctan\left( \frac{2}{\sqrt{\nu}} \left(\frac{1}{16} - \left(x-\frac{1}{2}\right)^2 - \left(y-\frac{1}{2}\right)^2 \right)   \right)\right).
\end{aligned}
\end{equation}

The solution $u$ develops steep gradients along the circle $0.0625 - (x-0.5)^2 - (y - 0.5)^2 =0$ of width $\mathcal{O}(\sqrt{\mu}).$ In Figure \ref{fig:hump_layer_convergence}, we present the experimental order of convergence for \eqref{conv_diff} with the solution given by \eqref{non_linear_system_f_sol_boundary_layer} and $\mu=10^{-3},\,10^{-4},\,10^{-5}$ in both the $L^2$-norm and the energy norm for the standard FEM and the AFC scheme. In particular, for all presenting $\mu,$ the results for $\alpha=0.2$ and $\alpha=0.5$ are very close. Further, in both cases, the AFC scheme achieves the optimal experimental order of convergence in both norms.

\begin{figure}
\centering
\begin{subfigure}[b]{0.32\textwidth}
\begin{tikzpicture}[scale=0.6]
\begin{axis}[
    xlabel={$h_0$},
    ylabel={Error},
    xmode=log,
    ymode=log,
    legend entries={
		{$\|\cdot\|_{L^2},\,\alpha=0.2$}, 
        {$\tribar \cdot \tribar,\,\alpha=0.2$}, 
        {$\mathcal{O}(h^2)$},
        {$\|\cdot\|_{L^2},\,\alpha=0.5$}, 
        {$\tribar \cdot \tribar,\,\alpha=0.5$},            
        {$\mathcal{O}(h)$}               
    },
    legend style={font=\footnotesize, at={(1,1)}, anchor=north west},
    title={$\mu=10^{-3}$},
    tick label style={font=\scriptsize},
    xlabel style={font=\small},
    ylabel style={font=\small},
]
\addplot[black, solid, mark=o, thick] 
table[x=h0, y=L2_mu1e3, skip coords between index={6}{1000}] {data_hump_a2.dat};
\addplot[black, solid, mark=triangle*, thick] 
table[x=h0, y=H1_mu1e3, skip coords between index={6}{1000}] {data_hump_a2.dat};

\addplot[black, dashed, domain=3.1250e-03:1e-1] { 2 * x^2 };

\addplot[black, dashed, mark=*, thick] 
table[x=h0, y=L2_mu1e3, skip coords between index={6}{1000}] {data_hump_a5.dat};
\addplot[black, dashed, mark=triangle, thick] 
table[x=h0, y=H1_mu1e3, skip coords between index={6}{1000}] {data_hump_a5.dat};

\addplot[gray, dotted, thick, domain=3.1250e-03:1e-1] { x };

\end{axis}
\end{tikzpicture}
\end{subfigure}
\hspace{1.5cm}
\begin{subfigure}[b]{0.32\textwidth}
\centering
\begin{tikzpicture}[scale=0.6]
\begin{axis}[
    xlabel={$h_0$},
    ylabel={Error},
    xmode=log,
    ymode=log,
    legend entries={
		{$\|\cdot\|_{L^2},\,\alpha=0.2$}, 
        {$\tribar \cdot \tribar,\,\alpha=0.2$}, 
        {$\mathcal{O}(h^2)$},
        {$\|\cdot\|_{L^2},\,\alpha=0.5$}, 
        {$\tribar \cdot \tribar,\,\alpha=0.5$},            
        {$\mathcal{O}(h)$}      
    },
    legend style={font=\footnotesize, at={(1,1)}, anchor=north west},
    title={$\mu=10^{-4}$},
    tick label style={font=\scriptsize},
    xlabel style={font=\small},
    ylabel style={font=\small},
]
\addplot[black, mark=square*, thick] table[x=h0, y=L2_mu1e4] {data_hump_a2.dat};
\addplot[black, mark=diamond*, thick] table[x=h0, y=H1_mu1e4] {data_hump_a2.dat};

\addplot[black, dashed, domain=1.5625e-03:1e-1] {10 * x^2 };

\addplot[black, mark=square*, thick] table[x=h0, y=L2_mu1e4] {data_hump_a5.dat};
\addplot[black, mark=diamond*, thick] table[x=h0, y=H1_mu1e4] {data_hump_a5.dat};

\addplot[gray, dotted, thick, domain=1.5625e-03:1e-1] {5 * x };

\end{axis}
\end{tikzpicture}
\end{subfigure}
\begin{subfigure}[b]{0.32\textwidth}
\begin{tikzpicture}[scale=0.6]
\begin{axis}[
    xlabel={$h_0$},
    ylabel={Error},
    xmode=log,
    ymode=log,
    legend entries={
		{$\|\cdot\|_{L^2},\,\alpha=0.2$}, 
        {$\tribar \cdot \tribar,\,\alpha=0.2$}, 
        {$\mathcal{O}(h^2)$},
        {$\|\cdot\|_{L^2},\,\alpha=0.5$}, 
        {$\tribar \cdot \tribar,\,\alpha=0.5$},            
        {$\mathcal{O}(h)$}                    
    },
    legend style={font=\footnotesize, at={(1,1)}, anchor=north west},
    title={$\mu=10^{-5}$},
    tick label style={font=\scriptsize},
    xlabel style={font=\small},
    ylabel style={font=\small},
]
\addplot[black, mark=o, thick] table[x=h0, y=L2_mu1e5] {data_hump_a2.dat};
\addplot[black, mark=triangle*, thick] table[x=h0, y=H1_mu1e5] {data_hump_a2.dat};

\addplot[black, dashed, domain=1.5625e-03:1e-1] {50 * x^2 };

\addplot[black, mark=o, thick] table[x=h0, y=L2_mu1e5] {data_hump_a5.dat};
\addplot[black, mark=triangle*, thick] table[x=h0, y=H1_mu1e5] {data_hump_a5.dat};

\addplot[gray, dotted, thick, domain=1.5625e-03:1e-1] { 10 * x };

\end{axis}
\end{tikzpicture}
\end{subfigure}
\caption{Results for \eqref{non_linear_system_f_sol_boundary_layer} using AFC for $\alpha=0.2,\,\alpha=0.5$: (top left) for $\mu = 10^{-3}$, (top right) for $\mu = 10^{-4}$ and (bottom) for $\mu = 10^{-5}.$}
\label{fig:hump_layer_convergence}
\end{figure}
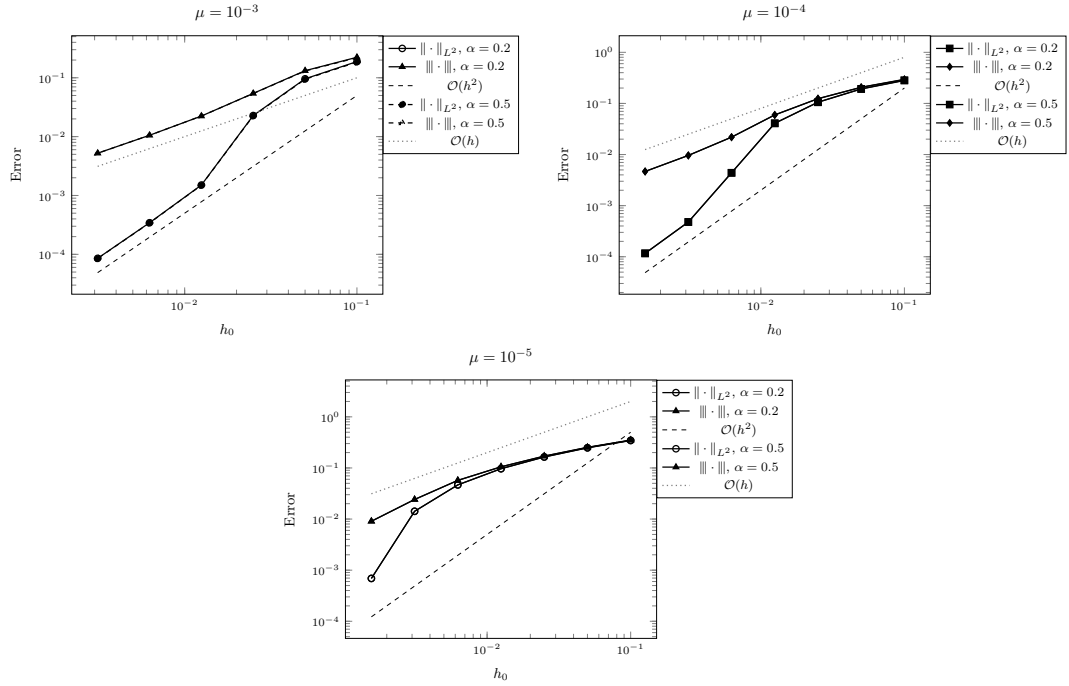

\subsection{Convergence study with solution with circular interior layer}

We consider the following time fractional convection--diffusion--reaction equation on  $\Omega = [0,1]^2,$ where we choose $T=1,$ and for small values of $\mu,$ we compute the corresponding function $G,$ so as to have the 
\begin{equation}\label{non_linear_system_f_sol_boundary_layer2}
\begin{aligned}
u(x,y,t)  =  \frac{t^2}{1 + e^{-\frac{\sqrt{(x-1)^2 + (y-1)^2} - 0.7}{\sqrt{\mu}}}}.
\end{aligned}
\end{equation}
The solution $u$ develops a circular interior layer along the circles $(x-1)^2 + (y-1)^2 = 0.7^2,$ of width $\mathcal{O}(\sqrt{\mu}).$ A similar test problem appear in \cite[Example 3]{wang2021}, in the context of steady--state optimal control problem that governed by a convection--diffusion equation.

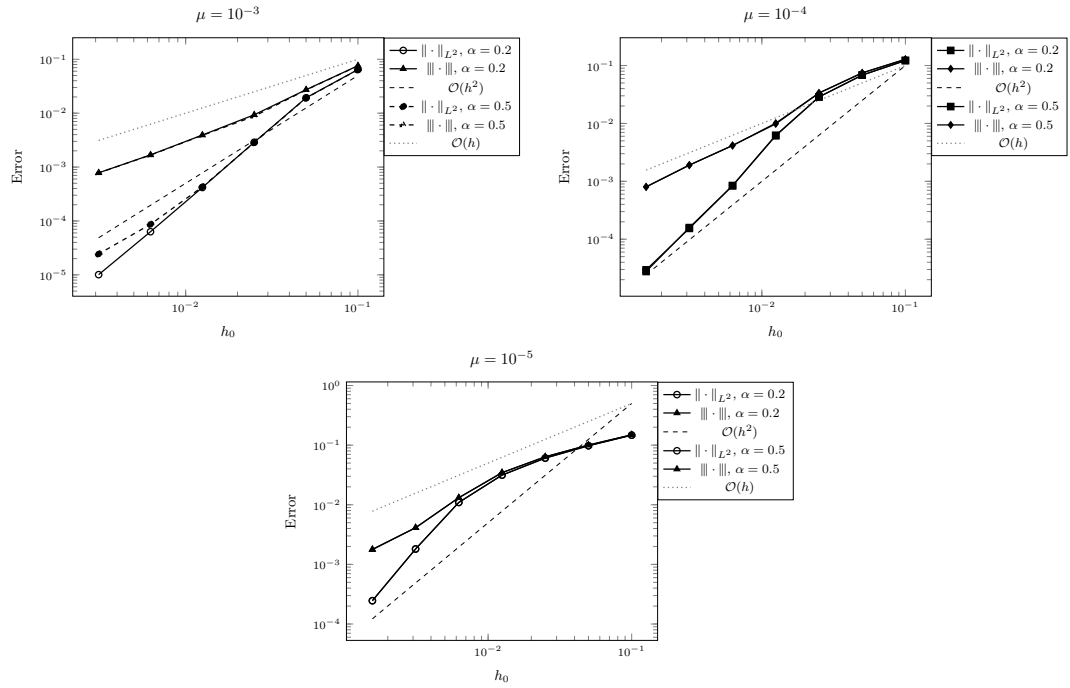
\begin{figure}
\centering
\begin{subfigure}[b]{0.32\textwidth}
\begin{tikzpicture}[scale=0.6]
\begin{axis}[
    xlabel={$h_0$},
    ylabel={Error},
    xmode=log,
    ymode=log,
    legend entries={
		{$\|\cdot\|_{L^2},\,\alpha=0.2$}, 
        {$\tribar \cdot \tribar,\,\alpha=0.2$}, 
        {$\mathcal{O}(h^2)$},
        {$\|\cdot\|_{L^2},\,\alpha=0.5$}, 
        {$\tribar \cdot \tribar,\,\alpha=0.5$},            
        {$\mathcal{O}(h)$}
    },
    legend style={font=\footnotesize, at={(1,1)}, anchor=north west},
    title={$\mu=10^{-3}$},
    tick label style={font=\scriptsize},
    xlabel style={font=\small},
    ylabel style={font=\small},
]

\addplot[black, solid, mark=o, thick] 
table[x=h0, y=L2_mu1e3, skip coords between index={6}{1000}] {data_interior_layer_a2.dat};
\addplot[black, solid, mark=triangle*, thick] 
table[x=h0, y=H1_mu1e3, skip coords between index={6}{1000}] {data_interior_layer_a2.dat};

\addplot[black, dashed, domain=3.1250e-03:1e-1] { 5 * x^2 };


\addplot[black, dashed, mark=*, thick] 
table[x=h0, y=L2_mu1e3, skip coords between index={6}{1000}] {data_interior_layer_a5.dat};
\addplot[black, dashed, mark=triangle, thick] 
table[x=h0, y=H1_mu1e3, skip coords between index={6}{1000}] {data_interior_layer_a5.dat};

\addplot[gray, dotted, thick, domain=3.1250e-03:1e-1] { x };

 
\end{axis}
\end{tikzpicture}
\end{subfigure}
\hspace{1.5cm}
\begin{subfigure}[b]{0.32\textwidth}
\centering
\begin{tikzpicture}[scale=0.6]
\begin{axis}[
    xlabel={$h_0$},
    ylabel={Error},
    xmode=log,
    ymode=log,
    legend entries={
		{$\|\cdot\|_{L^2},\,\alpha=0.2$}, 
        {$\tribar \cdot \tribar,\,\alpha=0.2$}, 
        {$\mathcal{O}(h^2)$},
        {$\|\cdot\|_{L^2},\,\alpha=0.5$}, 
        {$\tribar \cdot \tribar,\,\alpha=0.5$},            
        {$\mathcal{O}(h)$}   
    },
    legend style={font=\footnotesize, at={(1,1)}, anchor=north west},
    title={$\mu=10^{-4}$},
    tick label style={font=\scriptsize},
    xlabel style={font=\small},
    ylabel style={font=\small},
]
\addplot[black, mark=square*, thick] table[x=h0, y=L2_mu1e4] {data_interior_layer_a2.dat};
\addplot[black, mark=diamond*, thick] table[x=h0, y=H1_mu1e4] {data_interior_layer_a2.dat};

\addplot[black, dashed, domain=1.5625e-03:1e-1] {20 * x^2 };

\addplot[black, mark=square*, thick] table[x=h0, y=L2_mu1e4] {data_interior_layer_a5.dat};
\addplot[black, mark=diamond*, thick] table[x=h0, y=H1_mu1e4] {data_interior_layer_a5.dat};

\addplot[gray, dotted, thick, domain=1.5625e-03:1e-1] {8 * x };

\end{axis}
\end{tikzpicture}
\end{subfigure}
\begin{subfigure}[b]{0.32\textwidth}
\begin{tikzpicture}[scale=0.6]
\begin{axis}[
    xlabel={$h_0$},
    ylabel={Error},
    xmode=log,
    ymode=log,
    legend entries={
		{$\|\cdot\|_{L^2},\,\alpha=0.2$}, 
        {$\tribar \cdot \tribar,\,\alpha=0.2$}, 
        {$\mathcal{O}(h^2)$},
        {$\|\cdot\|_{L^2},\,\alpha=0.5$}, 
        {$\tribar \cdot \tribar,\,\alpha=0.5$},            
        {$\mathcal{O}(h)$}                 
    },
    legend style={font=\footnotesize, at={(1,1)}, anchor=north west},
    title={$\mu=10^{-5}$},
    tick label style={font=\scriptsize},
    xlabel style={font=\small},
    ylabel style={font=\small},
]
\addplot[black, mark=o, thick] table[x=h0, y=L2_mu1e5] {data_interior_layer_a2.dat};
\addplot[black, mark=triangle*, thick] table[x=h0, y=H1_mu1e5] {data_interior_layer_a2.dat};

\addplot[black, dashed, domain=1.5625e-03:1e-1] {50 * x^2 };

\addplot[black, mark=o, thick] table[x=h0, y=L2_mu1e5] {data_interior_layer_a5.dat};
\addplot[black, mark=triangle*, thick] table[x=h0, y=H1_mu1e5] {data_interior_layer_a5.dat};

\addplot[gray, dotted, thick, domain=1.5625e-03:1e-1] {20 * x };

\end{axis}
\end{tikzpicture}
\end{subfigure}
\caption{Results for \eqref{non_linear_system_f_sol_boundary_layer2} using AFC for $\alpha=0.2,\,\alpha=0.5$: (top left) for $\mu = 10^{-3}$, (top right) for $\mu = 10^{-4}$ and (bottom) for $\mu = 10^{-5}.$}
\label{fig:circular_layer_convergence}
\end{figure}

In Figure \ref{fig:circular_layer_convergence} we compute the errors in $L^2$ and energy norm for $\mu = 10^{-3},\,10^{-4},\,10^{-5}$ where the solution is given by \eqref{non_linear_system_f_sol_boundary_layer2}. Again, the AFC method achieves the optimal experimental order of convergence in both norms for both $\alpha=0.2$ and $\alpha=0.5.$

\subsection{Convergence study with solution with interior layer}

We consider the following time fractional convection--diffusion--reaction equation on  $\Omega = (0,1)^2,$ where we choose $T=1,$ and for small values of $\mu,$ we compute the corresponding function $G,$ so as to have the 
\begin{equation}\label{non_linear_system_f_sol_boundary_layer3}
\begin{aligned}
u(x,y,t)  = t^2\arctan\left(\frac{1}{\sqrt{\mu}}\left(-\frac{1}{2}x + y - \frac{1}{4}\right) \right)
\end{aligned}
\end{equation}
The solution $u$ develops a interior layer along the line $-0.5x + y - 0.25 = 0,$ of width $\mathcal{O}(\sqrt{\mu}).$  In a similar manner to the previous numerical examples, the Figure \ref{fig:interior_layer_convergence} present the experimental order of convergence for \eqref{conv_diff} with the solution given by \eqref{non_linear_system_f_sol_boundary_layer3} and $\mu=10^{-3},\,10^{-4},\,10^{-5}$ in both $L^2$-norm and energy norm for the standard FEM and the AFC scheme. Again, the AFC scheme achieves the optimal experimental order of convergence in both norms. 

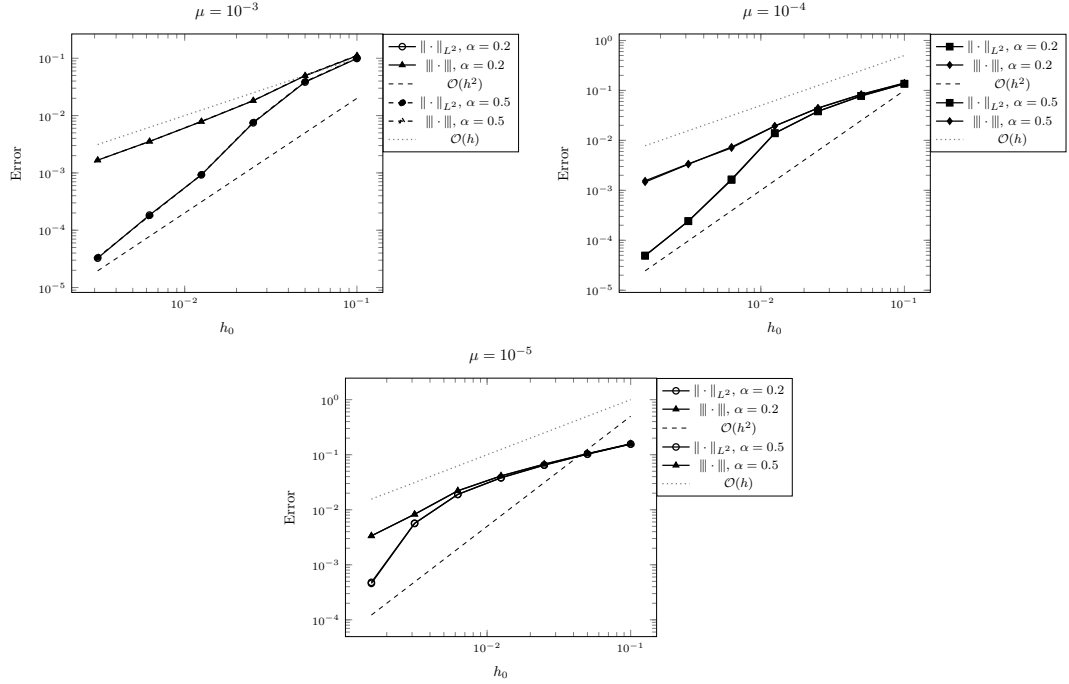
\begin{figure}
\centering
\begin{subfigure}[b]{0.32\textwidth}
\begin{tikzpicture}[scale=0.6]
\begin{axis}[
    xlabel={$h_0$},
    ylabel={Error},
    xmode=log,
    ymode=log,
    legend entries={
		{$\|\cdot\|_{L^2},\,\alpha=0.2$}, 
        {$\tribar \cdot \tribar,\,\alpha=0.2$}, 
        {$\mathcal{O}(h^2)$},
        {$\|\cdot\|_{L^2},\,\alpha=0.5$}, 
        {$\tribar \cdot \tribar,\,\alpha=0.5$},            
        {$\mathcal{O}(h)$}                
    },
    legend style={font=\footnotesize, at={(1,1)}, anchor=north west},
    title={$\mu=10^{-3}$},
    tick label style={font=\scriptsize},
    xlabel style={font=\small},
    ylabel style={font=\small},
]
\addplot[black, solid, mark=o, thick] 
table[x=h0, y=L2_mu1e3, skip coords between index={6}{1000}] {data_circular_layer_a2.dat};
\addplot[black, solid, mark=triangle*, thick] 
table[x=h0, y=H1_mu1e3, skip coords between index={6}{1000}] {data_circular_layer_a2.dat};

\addplot[black, dashed, domain=3.1250e-03:1e-1] { 5 * x^2 };

\addplot[black, dashed, mark=*, thick] 
table[x=h0, y=L2_mu1e3, skip coords between index={6}{1000}] {data_circular_layer_a5.dat};
\addplot[black, dashed, mark=triangle, thick] 
table[x=h0, y=H1_mu1e3, skip coords between index={6}{1000}] {data_circular_layer_a5.dat};

\addplot[gray, dotted, thick, domain=3.1250e-03:1e-1] { x };

\end{axis}
\end{tikzpicture}
\end{subfigure}
\hspace{1.5cm}
\begin{subfigure}[b]{0.32\textwidth}
\centering
\begin{tikzpicture}[scale=0.6]
\begin{axis}[
    xlabel={$h_0$},
    ylabel={Error},
    xmode=log,
    ymode=log,
    legend entries={
		{$\|\cdot\|_{L^2},\,\alpha=0.2$}, 
        {$\tribar \cdot \tribar,\,\alpha=0.2$}, 
        {$\mathcal{O}(h^2)$},
        {$\|\cdot\|_{L^2},\,\alpha=0.5$}, 
        {$\tribar \cdot \tribar,\,\alpha=0.5$},            
        {$\mathcal{O}(h)$}       
    },
    legend style={font=\footnotesize, at={(1,1)}, anchor=north west},
    title={$\mu=10^{-4}$},
    tick label style={font=\scriptsize},
    xlabel style={font=\small},
    ylabel style={font=\small},
]
\addplot[black, mark=square*, thick] table[x=h0, y=L2_mu1e4] {data_circular_layer_a2.dat};
\addplot[black, mark=diamond*, thick] table[x=h0, y=H1_mu1e4] {data_circular_layer_a2.dat};

\addplot[black, dashed, domain=1.5625e-03:1e-1] {10 * x^2 };

\addplot[black, mark=square*, thick] table[x=h0, y=L2_mu1e4] {data_circular_layer_a5.dat};
\addplot[black, mark=diamond*, thick] table[x=h0, y=H1_mu1e4] {data_circular_layer_a5.dat};

\addplot[gray, dotted, thick, domain=1.5625e-03:1e-1] { x };

\end{axis}
\end{tikzpicture}
\end{subfigure}
\begin{subfigure}[b]{0.32\textwidth}
\begin{tikzpicture}[scale=0.6]
\begin{axis}[
    xlabel={$h_0$},
    ylabel={Error},
    xmode=log,
    ymode=log,
    legend entries={
		{$\|\cdot\|_{L^2},\,\alpha=0.2$}, 
        {$\tribar \cdot \tribar,\,\alpha=0.2$}, 
        {$\mathcal{O}(h^2)$},
        {$\|\cdot\|_{L^2},\,\alpha=0.5$}, 
        {$\tribar \cdot \tribar,\,\alpha=0.5$},            
        {$\mathcal{O}(h)$}                       
    },
    legend style={font=\footnotesize, at={(1,1)}, anchor=north west},
    title={$\mu=10^{-5}$},
    tick label style={font=\scriptsize},
    xlabel style={font=\small},
    ylabel style={font=\small},
]
\addplot[black, mark=o, thick] table[x=h0, y=L2_mu1e5] {data_circular_layer_a2.dat};
\addplot[black, mark=triangle*, thick] table[x=h0, y=H1_mu1e5] {data_circular_layer_a2.dat};

\addplot[black, dashed, domain=1.5625e-03:1e-1] {50 * x^2 };

\addplot[black, mark=o, thick] table[x=h0, y=L2_mu1e5] {data_circular_layer_a5.dat};
\addplot[black, mark=triangle*, thick] table[x=h0, y=H1_mu1e5] {data_circular_layer_a5.dat};

\addplot[gray, dotted, thick, domain=1.5625e-03:1e-1] {5 * x };

\end{axis}
\end{tikzpicture}
\end{subfigure}
\caption{Results for \eqref{non_linear_system_f_sol_boundary_layer3} using AFC for $\alpha=0.2,\,\alpha=0.5$: (top left) for $\mu = 10^{-3}$, (top right) for $\mu = 10^{-4}$ and (bottom) for $\mu = 10^{-5}.$}
\label{fig:interior_layer_convergence}
\end{figure}

\section{Conclusions}

In this paper, we considered a time fractional convection--diffusion--reaction equation on a bounded domain $\Omega\subset\mathbb{R}^2$ and proposed a stabilized finite element scheme based on the algebraic flux correction method for its numerical approximation. That stabilized scheme proved that satisfies the discrete maximum principle. For nonsmooth initial data, we derived $L^2$-error estimates for the stabilized semidiscrete scheme under a restriction on $\alpha$. We also performed several numerical experiments for problems involving solutions with layers. The obtained results demonstrate that the proposed stabilized scheme preserves the optimal convergence order.
 
\bigskip
\bibliographystyle{plain} 
\bibliography{ref}

\end{document}